\documentclass[12pt,english]{paper}
\usepackage[T1]{fontenc}
\usepackage[latin9]{inputenc}
\usepackage[a4paper]{geometry}
\usepackage{amsmath}
\usepackage{amsthm}
\usepackage{amssymb}
\usepackage{rotating}

\makeatletter

\providecommand{\tabularnewline}{\\}

\newcommand{\lyxaddress}[1]{
	\par {\raggedright #1
	\vspace{1.4em}
	\noindent\par}
}
\theoremstyle{plain}
\newtheorem{thm}{\protect\theoremname}
\theoremstyle{definition}
\newtheorem*{defn*}{\protect\definitionname}
\theoremstyle{plain}
\newtheorem*{fact*}{\protect\factname}
\theoremstyle{plain}
\newtheorem{prop}{\protect\propositionname}
\theoremstyle{remark}
\newtheorem{rem}{\protect\remarkname}
\theoremstyle{plain}
\newtheorem{lem}{\protect\lemmaname}
\theoremstyle{plain}
\newtheorem{cor}{\protect\corollaryname}
\theoremstyle{definition}
 \newtheorem{exam}{\protect\examplename}
\theoremstyle{definition}
\newtheorem*{example*}{\protect\examplename}

\newcommand{\R}{\mathbb{R}}
\newcommand{\Natu}{\mathbb{N}}

\newcommand{\im}{\operatorname{im}}

\newcommand{\diag}{\operatorname{diag}}

\newcommand{\dt}{{\rm d}}

\makeatother

\usepackage{babel}
\providecommand{\corollaryname}{Corollary}
\providecommand{\definitionname}{Definition}
\providecommand{\examplename}{Example}
\providecommand{\factname}{Fact}
\providecommand{\lemmaname}{Lemma}
\providecommand{\propositionname}{Proposition}
\providecommand{\remarkname}{Remark}
\providecommand{\theoremname}{Theorem}

\begin{document}
\title{On the sensitivity of implementations of a least-squares collocation
method for linear higher-index differential-algebraic equations}
\author{Michael Hanke}
\maketitle

\lyxaddress{Department of Mathematics, KTH Royal Institute of Technology, S-100
44 Stockholm, Sweden}
\begin{abstract}
The present paper continues our investigation of an implementation
of a least-squares collocation method for higher-index differential-algebraic
equations. In earlier papers, we were able to substantiate the choice
of basis functions and collocation points for a robust implementation
as well as algorithms for the solution of the discrete system. The
present paper is devoted to an analytic estimation of condition numbers
for different components of an implementation. We present error estimations,
which show the sources for the different errors. 
\end{abstract}

\section{Introduction}

In a series of papers\cite{HMTWW,HMT,HM,HM1}, we were developing
a new method for solving higher-index differential algebraic equations
(DAEs). In naturally given functional analytic settings, higher index
DAEs give rise to ill-posed problems \cite[Section 3.9]{Hanke1989,CRR}.
Motivated by the well-known method of least-squares, or discretization
on preimage space, for the approximation of ill-posed problems \cite{KaltOffter},
this approach has been adapted to the case of higher-index DAEs. In
particular, the ansatz spaces for the discrete least-squares problem
have been chosen to be piecewise polynomials. Additionally, the integrals
have been replaced by discrete versions based on simplified integration
rules, in the most simple approach by a version resembling well-known
collocation methods for solving boundary value problems for systems
of ordinary differential equations (ODEs). The latter, extremely simplified
version of the approach proposed in \cite{KaltOffter}, has been motivated
by the success of collocations methods for ODEs. This connection led
us to coin the notion \emph{least-squares collocation method} and
calling the integration nodes also as \emph{collocation points}.

For our method, a number of convergence results both for linear and
nonlinear DAEs have been proven. Even our first attempts showed surprisingly
accurate results when applying the method to some linear examples
\cite{HMTWW}. More recently, we investigated the algorithmic ingredients
of the method in more detail \cite{HM_Part1,HM_Part2}. Not surprisingly,
the basis representation and the choice of the integration nodes showed
an important influence on the accuracy of the method.

The present note is intended to further quantify the conditioning
of the individual ingredients of the implementation of the proposed
method and to better understand the (high) accuracy of the computational
results obtained so far. Taking the ill-posedness of higher-index
DAEs into account we expect very sensitive discrete problems for sufficiently
fine discretizations.

The practical implementation of a projection method consists of two
steps for a given approximation space $X_{\pi}$: Choice of a basis
and formulation and solution of the arising discrete system by a suitable
method. This in turn gives rise to two different operators, the first
being the \emph{representation map} connecting the elements of $x\in X_{\pi}$
with their vector of coefficients with respect to the chosen basis.
The other operator is the discrete version of the the least-squares
collocation method that becomes a linearly equality constrained linear
least-squares problem in our case. Both operators are investigated
in detail both analytically and numerically.

In particular, qualitative and quantitative estimations for the condition
numbers and norms of the representation map are proven for bases whose
usefulness in the present applications has been established earlier
\cite{HM_Part1,HM_Part2}.

For the constrained linear least-squares problem, a number of perturbation
results are well-known, e.g., \cite{We73,Elden80,CoxHig99}. However,
in the present application, the constraints play a special role: In
the usual choices of the basis functions, some coefficient vectors
do not represent a function in the approximation space. A coefficient
vector represents a function in the approximation space if and only
if the constraints are fulfilled. Therefore, a new error estimation
is derived, which takes care of the exceptional role of the constraints.
The important ingredients in this estimate are the condition number
of the constraints and a restricted condition number for the least-squares
functional. For the former, a complete analytical characterization
for the chosen bases is provided. In a number of numerical examples,
values for the restricted condition number are presented.

Based on the properties of the constraints, a new projection procedure
of inconsistent approximations onto the approximation space is proposed,
which is simple and easily implementable.

In Section 2, the least-squares method for approximating linear DAEs
is introduced and the representation map is constructed. Section 3
is devoted to an in-depth investigation of the representation map.
Then we derive a perturbation result for constrained linear least-squares
problem in Section 4. Numerical examples for the condition numbers
of the different ingredients are given in Section 5. Section 6 contains
some conclusions. The algorithm for projecting inconsistent coefficient
vectors is derived in Appendix~\ref{sec:Projections}.

\section{The problem setting}

\subsection{The discrete functional}

In this section, we repeat the problem setting from \cite{HM_Part1}
for the readers convenience. Consider a linear boundary-value problem
for a DAE with properly involved derivative, 
\begin{align}
A(t)(Dx)'(t)+B(t)x(t) & =q(t),\quad t\in[a,b],\label{eq:DAE}\\
G_{a}x(a)+G_{b}x(b) & =d.\label{eq:BC}
\end{align}
with $[a,b]\subset\R$ being a compact interval, $D=[I\;0]\in\R^{k\times m}$,
$k<m$, with the identity matrix $I\in\R^{k\times k}$. Furthermore,
$A(t)\in\R^{m\times k}$, $B(t)\in\R^{m\times m}$, and $q(t)\in\R^{m}$
are assumed to be sufficiently smooth with respect to $t\in[a,b]$.
Moreover, $G_{a},G_{b}\in\R^{l_{\textrm{dyn}}\times m}$. Thereby,
$l_{\textrm{dyn}}$ is the dynamical degree of freedom of the DAE,
that is, the number of free parameters that can be fixed by initial
and boundary conditions. We assume further that $\ker D\subseteq\ker G_{a}$
and $\ker D\subseteq\ker G_{b}$.

Unlike regular ODEs where $l_{\textrm{dyn}}=k=m$, for DAEs it holds
that $0\leq l_{\textrm{dyn}}\leq k<m$, in particular, $l_{\textrm{dyn}}=k$
for index-one DAEs, $l_{\textrm{dyn}}<k$ for higher-index DAEs, and
$l_{\textrm{dyn}}=0$ can certainly happen.

The appropriate space for looking for solutions of (\ref{eq:DAE})-(\ref{eq:BC})
is (cf \cite{HMT})
\[
H_{D}^{1}(a,b):=\{x\in L^{2}((a,b),\R^{m}:Dx\in H^{1}((a,b),\R^{m}\}.
\]

Let $\mathfrak{P}_{K}$ denote the set of all polynomials of degree
less than or equal to $K\geq0$. Given the partition $\pi$, 
\begin{equation}
\pi:\quad a=t_{0}<t_{1}<\cdots<t_{n}=b,\label{eq:mesh}
\end{equation}
with the stepsizes $h_{j}=t_{j}-t_{j-1}$, $h=\max_{1\leq j\leq n}h_{j}$,
and $h_{\min}=\min_{1\leq j\leq n}h_{j}$. Let $C_{\pi}([a,b],\R^{m})$
denote the space of piecewise continuous functions having breakpoints
merely at the meshpoints of the partition $\pi$. Let $N\geq1$ be
a fixed integer. We are looking for an approximate solution of our
boundary value problem from the ansatz space $X_{\pi}\subset H_{D}^{1}(a,b)$,
\begin{align}
X_{\pi} & =\{x\in C_{\pi}([a,b],\R^{m}):Dx\in C([a,b],\R^{k}),\nonumber \\
 & x_{\kappa}\lvert_{[t_{j-1},t_{j})}\in\mathfrak{P}_{N},\,\kappa=1,\ldots,k,\quad x_{\kappa}\lvert_{[t_{j-1},t_{j})}\in\mathfrak{P}_{N-1},\,\kappa=k+1,\ldots,m,\;j=1,\ldots,n\}.\label{eq:Xn}
\end{align}
The continuous version of the least-squares method reads: Find an
$x_{\pi}\in X_{\pi}$ that minimizes the functional 
\begin{equation}
\Phi(x)=\int_{a}^{b}|A(t)(Dx)'(t)+B(t)x(t)-q(t)|^{2}\dt t+|G_{a}x(a)+G_{b}x(b)-d|^{2}.\label{eq:Phi}
\end{equation}
Here and in the following, $|\cdot|$ denotes the Euclidean norm in
the corresponding spaces $\R^{\alpha}$ for the appropriate $\alpha$.
Let $\langle\cdot,\cdot\rangle$ denote the scalar product in $\R^{\alpha}$.

The functional values $\Phi(x)$, which are needed when minimizing
for $x\in X_{\pi}$, cannot be evaluated exactly and the integral
must be discretized accordingly. Taking into account that the boundary-value
problem is ill-posed in the higher index case, perturbations of the
functional may have a serious influence on the error of the approximate
least-squares solution or even prevent convergence towards the exact
solution. Therefore, careful approximations of the integral in $\Phi$
are required. We take over the options provided in \cite{HM_Part1},
in which $M\geq N+1$ so-called collocation points 
\begin{equation}
0\leq\rho_{1}<\cdots<\rho_{M}\leq1.\label{eq:nodes}
\end{equation}
are used, and further, on the subintervals of the partition $\pi$,
\[
t_{ji}=t_{j-1}+\rho_{i}h_{j},\quad i=1,\ldots,M,\;j=1,\ldots,n.
\]
Introducing, for each $x\in X_{\pi}$ and $w(t)=A(t)(Dx)'(t)+B(t)x(t)-q(t)$,
the corresponding vector $W\in\R^{mMn}$ by 
\begin{equation}
W=\left[\begin{array}{c}
W_{1}\\
\vdots\\
W_{n}
\end{array}\right]\in\R^{mMn},\quad W_{j}=h_{j}^{1/2}\left[\begin{array}{c}
w(t_{j1})\\
\vdots\\
w(t_{jM})
\end{array}\right]\in\R^{mM},\label{eq:W}
\end{equation}
we turn to an approximate functional of the form 
\begin{align}
\Phi_{\pi,M}(x)=W^{T}\mathcal{L}W+|G_{a}x(a)+G_{b}x(b)-d|^{2},\quad x\in X_{\pi},\label{eq:PhiM}
\end{align}
with a positive definite symmetric matrix\footnote{$\otimes$ denotes the Kronecker product.}
\begin{align}
\mathcal{L}=\diag(L\otimes I_{m},\ldots,L\otimes I_{m}).\label{eq:L}
\end{align}
As detailed in \cite{HM_Part1}, we have different options for the
positive definite symmetric matrix $L\in\R^{M\times M}$, namely 
\begin{align}
L & =L^{C}=M^{-1}I_{M},\label{LC}\\
L & =L^{I}=\diag(\gamma_{1},\ldots,\gamma_{M}),\label{LI}\\
L & =L^{R}=(V^{-1})^{T}V^{-1},\label{LR}
\end{align}
see \cite[Section 3]{HM_Part1} for details concerning the selection
of the quadrature weights $\gamma_{1},\ldots,\gamma_{M}$ and the
construction of the mass matrix $V$. We emphasize that the matrices
$L^{C},L^{I},L^{R}$ depend only on $M$, the node sequence (\ref{eq:nodes}),
and the quadrature weights, but do not depend on the partition $\pi$
and its stepsizes at all.

In the context of the numerical experiments below, we denote each
of the different versions of the functional by $\Phi_{\pi,M}^{C}$,
$\Phi_{\pi,M}^{I}$, and $\Phi_{\pi,M}^{R}$, respectively. The following
convergence result is known \cite[Theorem 2]{HM_Part1}:
\begin{thm}
\label{thm:Convergence}Let the DAE (\ref{eq:DAE}) be regular with
index $\mu\in\Natu$ and let the boundary condition (\ref{eq:BC})
be accurately stated. Let $x_{*}$ be a solution of the boundary value
problem (\ref{eq:DAE})--(\ref{eq:BC}), and let $A,B,q$ and also
$x_{*}$ be sufficiently smooth.

Let all partitions $\pi$ be such that $h/h_{min}\leq\rho$, with
a global constant $\rho$. Then, with 
\[
M\geq N+\mu,
\]
the following statements are true: 
\begin{description}
\item [{\rm (1)}] For sufficient fine partitions $\pi$ and each sequence
of arbitrarily placed nodes (\ref{eq:nodes}), there exists exactly
one $x_{\pi}^{R}\in X_{\pi}$ minimizing the functional $\Phi_{\pi,M}^{R}$
on $X_{\pi}$, and 
\begin{align*}
\|x_{\pi}^{R}-x_{\ast}\|_{H_{D}^{1}(a,b)}\leq C_{R}h^{N-\mu+1}.
\end{align*}
\item [{\rm (2)}] For each integration rule related to the interval $[0,1]$,
with $M$ nodes (\ref{eq:nodes}) and positive weights $\gamma_{1},\ldots,\gamma_{M}$,
that is exact for polynomials with degree less than or equal to $2M-2$,
and sufficient fine partitions $\pi$, there exists exactly one $x_{\pi}^{I}\in X_{\pi}$
minimizing the functional $\Phi_{\pi,M}^{I}$ on $X_{\pi}$, and $x_{\pi}^{I}=x_{\pi}^{R}$,
thus 
\begin{align*}
\|x_{\pi}^{I}-x_{\ast}\|_{H_{D}^{1}(a,b)}\leq C_{R}h^{N-\mu+1}.
\end{align*}
\end{description}
\end{thm}
A corresponding result for $\Phi_{\pi,M}^{C}$ is not known. Numerical
tests showed excellent convergence results even for cases not covered
by Theorem~\ref{thm:Convergence}. This holds in particular for any
$M\geq N+1$ tested in all three cases of the functional $\Phi_{\pi,M}$.
Thus, $M=N+1$ seems to be the preferable choice.

\subsection{A basis representation of $\Phi_{\pi,M}$}

By choosing an appropriate basis for $X_{\pi}$, the minimization
of the functional (\ref{eq:PhiM}) will be reduced to a minimization
problem for the coefficients of the elements $x\in X_{\pi}$. For
the subsequent considerations, it is appropriate to introduce the
space
\begin{align}
\tilde{X}_{\pi} & =\{x\in C_{\pi}([a,b],\R^{m}):\nonumber \\
 & x_{\kappa}\lvert_{[t_{j-1},t_{j})}\in\mathfrak{P}_{N},\,\kappa=1,\ldots,k,\quad x_{\kappa}\lvert_{[t_{j-1},t_{j})}\in\mathfrak{P}_{N-1},\,\kappa=k+1,\ldots,m,\;j=1,\ldots,n\}.\label{eq:Xn-1-1}
\end{align}
In particular, the elements $x$ of $\tilde{X}_{\pi}$ are no longer
required to have continuous components $Dx$. Obviously, it holds
$X_{\pi}\subseteq\tilde{X}_{\pi}$. In general, $\tilde{X}_{\pi}$
is not a subspace of $H_{D}^{1}(a,b)$. However, it holds
\begin{align*}
X_{\pi} & =\{x\in\tilde{X}_{\pi}:x_{\kappa}\in C[a,b],\quad\kappa=1,\ldots,k\}\\
 & =\tilde{X}_{\pi}\cap H_{D}^{1}(a,b).
\end{align*}
Based on the analysis in \cite[Section 4]{HM_Part1} we provide a
basis of the ansatz space $\tilde{X}_{\pi}$ to begin with. Assume
that $\{p_{0},\ldots,p_{N-1}\}$ is a basis of $\mathfrak{P}_{N-1}$
defined on the reference interval $[0,1]$. Then, $\{\bar{p}_{0},\ldots,\bar{p}_{N}\}$
given by 
\begin{equation}
\bar{p}_{i}(\tau)=\begin{cases}
1, & i=0,\\
\int_{0}^{\tau}p_{i-1}(\sigma){\rm d}\sigma, & i=1,\ldots,N,\quad\tau\in[0,1],
\end{cases}\label{eq:diffBasis}
\end{equation}
form a basis of $\mathfrak{P}_{N}$. The transformation to the interval
$(t_{j-1},t_{j})$ of the partition $\pi$ (\ref{eq:mesh}) yields
\begin{align}
p_{ji}(t)=p_{i}((t-t_{j-1})/h_{j}),\quad\bar{p}_{ji}(t)=h_{j}\bar{p}_{i}((t-t_{j-1})/h_{j}).\label{eq:scaledp}
\end{align}
and in particular 
\begin{align*}
\bar{p}_{ji}(t_{j-1}) & =h_{j}\bar{p}_{i}(0)=h_{j}\begin{cases}
1, & i=0,\\
0, & i=1,\ldots,N,
\end{cases}\\
\bar{p}_{ji}(t_{j}) & =h_{j}\bar{p}_{i}(1)=h_{j}\begin{cases}
1, & i=0,\\
\int_{0}^{1}p_{i-1}(\sigma){\rm d}\sigma, & i=1,\ldots,N.
\end{cases}
\end{align*}
Next we form the matrix functions 
\begin{align*}
\bar{\mathcal{P}}_{j}=\begin{bmatrix}\bar{p}_{j0} & \ldots & \bar{p}_{jN}\end{bmatrix}:[t_{j-1},t_{j}]\rightarrow\R^{1\times(N+1)},\quad\mathcal{P}_{j}=\begin{bmatrix}p_{j0} & \ldots & p_{j,N-1}\end{bmatrix}:[t_{j-1},t_{j}]\rightarrow\R^{1\times N},
\end{align*}
such that 
\begin{align}
\bar{\mathcal{P}}_{j}(t_{j-1}) & =h_{j}\begin{bmatrix}1 & 0 & \ldots & 0\end{bmatrix},\quad j=1,\ldots,n,\label{eq:barPj}\\
\bar{\mathcal{P}}_{j}(t_{j}) & =h_{j}\begin{bmatrix}1 & \int_{0}^{1}p_{0}(\sigma){\rm d}\sigma & \ldots & \int_{0}^{1}p_{N-1}(\sigma){\rm d}\sigma\end{bmatrix},\quad j=1,\ldots,n.\label{eq:barPj2}
\end{align}
Following the discussions in \cite{HM_Part1}, the following bases
are suitable in applications:
\begin{description}
\item [{Legendre~basis}] Let $P_{i}$ denote the Legendre polynomials.
Then, $p_{i}$ is chosen to be the shifted Legendre polynomial, that
is
\[
p_{i}(\tau)=P_{i}(2\tau-1),\quad i=0,1,\ldots.
\]
 
\item [{Modified~Legendre~basis}] In this case, we set
\[
\bar{p}_{0}(\tau)=1,\quad\bar{p}_{i}(\tau)=P_{i}(2\tau-1)-(-1)^{i},\quad i=1,2,\ldots,
\]
such that $p_{i}=\bar{p}_{i+1}'$, $i=0,1,\ldots$. This basis has
not been considered in \cite{HM_Part1}, but later experiments indicated
its usefulness. This is supported by considerations later below.
\item [{Chebyshev~basis}] Let $T_{i}$ denote the Chebyshev polynomials
of the first kind. Then we define
\[
p_{i}(\tau)=T_{i}(2\tau-1),\quad i=0,1,\ldots.
\]
\item [{Runge-Kutta~basis}] Let $0<\tau_{1}<\cdots<\tau_{N}<1$ be interpolation
nodes. Then we set
\begin{equation}
p_{i}(\tau)=\frac{\prod_{\kappa\neq i+1}(\tau-\tau_{\kappa})}{\prod_{\kappa\neq i+1}(\tau_{i+1}-\tau_{\kappa})}.\label{eq:ChebNodes}
\end{equation}
The latter are the usual Lagrange interpolation polynomials. In the
implementation, it is advantageous to represent these polynomials
in terms of Chebyshev polynomials \cite{HM_Part1}. Of particular
use is the Runge-Kutta basis if the shifted Chebyshev nodes $\tau_{\kappa}=\frac{1}{2}\left(1+\cos\left(\frac{2\kappa-1}{2N}\pi\right)\right)$
are chosen as interpolation nodes.
\end{description}
\medskip{}

For $x\in\tilde{X}_{\pi}$ we use the denotations 
\begin{align*}
x(t)=x_{j}(t)=\begin{bmatrix}x_{j1}(t)\\
\vdots\\
x_{jm}(t)
\end{bmatrix}\in\R^{m},\quad Dx_{j}(t)=\begin{bmatrix}x_{j1}(t)\\
\vdots\\
x_{jk}(t)
\end{bmatrix}\in\R^{k},\quad t\in[t_{j-1},t_{j}).
\end{align*}
Then, we develop each $x_{j}$ componentwise 
\begin{align}
\begin{aligned}x_{j\kappa}(t) & =\sum_{l=0}^{N}c_{j\kappa l}\bar{p}_{jl}(t)=\bar{\mathcal{P}}_{j}(t)c_{j\kappa},\quad\kappa=1,\ldots,k,\\
x_{j\kappa}(t) & =\sum_{l=0}^{N-1}c_{j\kappa l}p_{jl}(t)=\mathcal{P}_{j}(t)c_{j\kappa},\quad\kappa=k+1,\ldots,m.
\end{aligned}
\label{eq:reprx}
\end{align}
with 
\begin{align*}
c_{j\kappa}=\begin{bmatrix}c_{j\kappa0}\\
\vdots\\
c_{j\kappa N}
\end{bmatrix}\in\R^{N+1},\quad\kappa=1,\ldots,k,\quad c_{j\kappa}=\begin{bmatrix}c_{j\kappa0}\\
\vdots\\
c_{j\kappa,N-1}
\end{bmatrix}\in\R^{N},\quad\kappa=k+1,\ldots,m.
\end{align*}
Introducing still 
\begin{align*}
\Omega_{j}(t)=\left[\begin{array}{cc}
I_{k}\otimes\bar{\mathcal{P}}_{j}(t) & \mathcal{O}_{1}\\
\mathcal{O}_{2} & I_{m-k}\otimes\mathcal{P}_{j}(t)
\end{array}\right]\in\R^{m\times(mN+k)},\quad c_{j}=\begin{bmatrix}c_{j1}\\
\vdots\\
c_{jm}
\end{bmatrix}\in\R^{mN+k},
\end{align*}
with $\mathcal{O}_{1}\in\R^{k\times kN}$ and $\mathcal{O}_{2}\in\R^{(m-k)\times(m-k)(N+1)}$
being matrices having only zero entries we represent, for $t\in I_{j}$,
$j=1,\ldots,n$, 
\begin{align}
x_{j}(t) & =\Omega_{j}(t)c_{j},\label{eq:xj}\\
Dx_{j}(t) & =D\Omega_{j}(t)c_{j}=\begin{bmatrix}I_{k}\otimes\bar{\mathcal{P}}_{j}'(t) & \;\mathcal{O}_{1}\end{bmatrix}c_{j}.\label{eq:Dxj}
\end{align}
Now we collect all coefficients $c_{j\kappa\,l}$ in the vector $c$,
\begin{align*}
c=\begin{bmatrix}c_{1}\\
\vdots\\
c_{n}
\end{bmatrix}\in\R^{n(mN+k)}.
\end{align*}

\begin{defn*}
The mapping $\mathcal{R}:\R^{n(mN+k)}\rightarrow\tilde{X}_{\pi}$
given by (\ref{eq:xj}) is called the \emph{representation map} of
$\tilde{X}_{\pi}$ with respect to the basis (\ref{eq:scaledp}).
\end{defn*}
\begin{fact*}
We observe that each $x\in\tilde{X}_{\pi}$ has a representation of
the kind (\ref{eq:xj}) and each function of the form (\ref{eq:xj})
is an element of $\tilde{X}_{\pi}$. Since $\dim\tilde{X}_{\pi}=n(mN+k)$,
$\mathcal{R}$ is a bijective mapping.
\end{fact*}
\medskip{}

Consider an element $x\in\tilde{X}_{\pi}$ with its representation
(\ref{eq:xj}). This element belongs to $X_{\pi}$ if and only if
its first $k$ components are continuous. Using the representation
(\ref{eq:reprx}) we see that $x\in X_{\pi}$ if and only if
\begin{equation}
\mathcal{C}c=0.\label{eq:discconstraint}
\end{equation}
where $\mathcal{C}\in\R^{k(n-1)\times n(mN+k)}$ and
\[
\mathcal{C}=\begin{bmatrix}I_{k}\otimes\bar{\mathcal{P}}_{1}(t_{1}) & \mathcal{O}_{1} & -I_{k}\otimes\bar{\mathcal{P}}{}_{2}(t_{1}) & \mathcal{O}_{1}\\
 &  & I_{k}\otimes\bar{\mathcal{P}}_{2}(t_{2}) & \mathcal{O}_{1} & -I_{k}\otimes\bar{\mathcal{P}}_{3}(t_{2}) & \mathcal{O}_{1}\\
 &  &  & \ddots &  & \ddots\\
\\
 &  &  &  & I_{k}\otimes\bar{\mathcal{P}}_{n-1}(t_{n-1}) & \mathcal{O}_{1} & -I_{k}\otimes\bar{\mathcal{P}}_{n}(t_{n-1}) & \mathcal{O}_{1}
\end{bmatrix}.
\]
Owing to the construction, $\mathcal{C}$ has full row rank, cf. (\ref{eq:barPj}),
(\ref{eq:barPj2}).
\begin{fact*}
Define $\tilde{\mathcal{R}}=\left.\mathcal{R}\right|_{\ker\mathcal{C}}$
be the restriction of the representation map $\mathcal{R}$ onto the
kernel $\ker\mathcal{C}$ of $\mathcal{C}$. Since $\mathcal{C}$
has full row rank, $\dim\ker\mathcal{C}=n(mN+k)-k(n-1)=nmN+k=\dim X_{\pi}$,
and $\mathcal{R}$ is injective, $\tilde{\mathcal{R}}$ is bijective.
In particular, it holds also $\tilde{\mathcal{R}}^{-1}=\left.\mathcal{R}^{-1}\right|_{\im\tilde{\mathcal{R}}}$.
\end{fact*}
The representations (\ref{eq:xj})--(\ref{eq:Dxj}) can be inserted
into the functional $\Phi_{\pi,M}$ (\ref{eq:PhiM}). The result becomes
a least-squares functional of the form
\begin{equation}
\varphi(c)=\lvert\mathcal{A}c-r\rvert_{\R^{nmM+l_{dyn}}}^{2}\rightarrow\min!\label{eq:discmin}
\end{equation}
where $\mathcal{A}$ has the structure
\[
\mathcal{A}=\left[\begin{array}{ccccc}
\mathcal{A}_{1} & 0 & \cdots &  & 0\\
0 & \ddots &  &  & \vdots\\
\vdots &  & \ddots\\
 &  &  & \ddots & 0\\
0 &  &  &  & \mathcal{A}_{n}\\
G_{a}\Omega_{1}(t_{0}) & 0 & \cdots & 0 & G_{b}\Omega_{n}(t_{n})
\end{array}\right]
\]
where $\mathcal{A}_{j}\in\R^{mM\times(mN+k)}$ and $G_{a}\Omega_{1}(t_{0}),G_{b}\Omega_{n}(t_{n})\in\R^{l_{\textrm{dyn}}\times(mN+k)}$.

So the discrete version of the least-squares method (\ref{eq:PhiM})
becomes the linear least-squares problem (\ref{eq:discmin}) under
the linear equality constraint (\ref{eq:discconstraint}).

Note that it holds $r\in\R^{nmM+l_{\textrm{dyn}}}$ and $\mathcal{A}\in\R^{(nmM+l_{\textrm{dyn}})\times n(mN+k)}$.
The matrices $\mathcal{A}$ and $\mathcal{C}$ are very sparse. More
details of the construction of $\mathcal{A}$ and $\mathcal{C}$ can
be found in \cite{HM_Part2}.

\subsection{Conditioning of the implementation\label{subsec:Algorithm}}

The implementation for solving the least-squares problem (\ref{eq:PhiM})
consists of the following steps:
\begin{enumerate}
\item Form $\mathcal{A}$, $\mathcal{C}$, and $r$.
\item Solve the constraint least-squares problem (\ref{eq:discmin})-(\ref{eq:discconstraint}).
\item Form the approximation $x_{\pi}$.
\end{enumerate}
What are the errors to be expected? Consider the individual steps:
\begin{enumerate}
\item The computation of $\mathcal{C}$ is not critical. Depending on the
chosen basis, the entries of $\mathcal{C}$ may be available analytically.
So we expect at most rounding errors for the representation of the
analytical data.\footnote{In the case of the Legendre and modified Legendre bases, all entries are integers weighted
by the stepsizes.} While the components of $\mathcal{A}$ corresponding to the boundary
conditions are only subject to truncation errors when representing
real numbers in floating point arithmetic, the DAE related entries
are subject to rounding errors as well as certain amplification factors
stemming from the multiplication by the square root of the matrix
$\mathcal{L}$ (\ref{eq:L}). The conditioning of the versions (\ref{LC})
and (\ref{LI}) is easy to infer while that of (\ref{LR}) has been
discussed extensively in \cite{HM_Part1}. Under reasonable assumptions
on the choice of collocation points, they are rather small.

Similar considerations apply to the computation of $r$.
\item This algorithmic step corresponds to the solution of a linearly constrained
linear least-squares problem. A number of classical perturbation results
are available, e.g., \cite{LaHa74,Elden80,CoxHig99}. Further below,
we represent a modified version that is taking into account the special
role that the equality constraint $\mathcal{C}c=0$ is playing in
our application.
\item This step is described by the representation map $\mathcal{R}$, which
assigns, to each solution $c$ of the previous step, the corresponding
solution $x_{\pi}=\mathcal{R}c$. If $c\in\ker\mathcal{C}$, it holds
$x_{\pi}\in X_{\pi}\subseteq H_{D}^{1}(a,b)$. However, due to the
errors made in the previous step, the condition $c\in\ker\mathcal{C}$
cannot be guaranteed such that $\mathcal{R}c\in\tilde{X}_{\pi}$ but
not necessarily $\mathcal{R}c\in H_{D}^{1}(a,b)$! In the next section,
we will discuss the properties of $\mathcal{R}$.
\end{enumerate}

\section{Properties of the representation map $\mathcal{R}$\label{subsec:Matrix-B}}

In the present section we will investigate the properties of the representation
map $\mathcal{R}:\R^{n(mN+k)}\rightarrow\tilde{X}_{\pi}$ in more
detail. Previously, we have established a representation of $\mathcal{R}$
on each subinterval, see (\ref{eq:xj}). We intend to derive a representation
of $\mathcal{R}^{-1}$. The main tool will be interpolation.

Choose two sets of interpolation nodes 
\begin{align}
0\leq\bar{\sigma}_{1}<\cdots<\bar{\sigma}_{N+1}\leq1\;\text{ and }\;0\leq\sigma_{1}<\cdots<\sigma_{N}\leq1,\label{Nodes}
\end{align}
and shifted ones 
\[
\bar{\tau}_{ji}=t_{j-1}+\bar{\sigma}_{i}h_{j},\quad\tau_{ji}=t_{j-1}+\sigma_{i}h_{j}
\]
 such that the integration formulae
\begin{align*}
\int_{0}^{1}f(\sigma){\rm d\sigma}\approx\sum_{i=1}^{N+1}\bar{\gamma}_{i}f(\bar{\sigma}_{i}),\quad\text{and }\;\int_{0}^{1}f(\sigma){\rm d\sigma}\approx\sum_{i=1}^{N}\gamma_{i}f(\sigma_{i})
\end{align*}
have positive weights and so that they are exact for polynomials up
to degree $2N$ and $2N-2$, respectively. With matrices 
\begin{align}
\bar{V}_{j} & =\begin{bmatrix}\bar{p}_{j0}(\bar{\tau}_{j1}) & \cdots & \bar{p}_{jN}(\bar{\tau}_{j1})\\
\vdots &  & \vdots\\
\bar{p}_{j0}(\bar{\tau}_{j,N+1}) & \cdots & \bar{p}_{jN}(\bar{\tau}_{j,N+1})
\end{bmatrix}=h_{j}\begin{bmatrix}\bar{p}_{0}(\bar{\sigma}_{1}) & \cdots & \bar{p}_{N}(\bar{\sigma}_{1})\\
\vdots &  & \vdots\\
\bar{p}_{0}(\bar{\sigma}_{N+1}) & \cdots & \bar{p}_{N}(\bar{\sigma}_{N+1})
\end{bmatrix}=:h_{j}\bar{V},\label{eq:Vb}\\
V_{j} & =\begin{bmatrix}p_{j0}(\tau_{j1}) & \cdots & p_{j,N-1}(\tau_{j1})\\
\vdots &  & \vdots\\
p_{j0}(\tau_{jN}) & \cdots & p_{j,N-1}(\tau_{jN})
\end{bmatrix}=\begin{bmatrix}p_{0}(\sigma_{1}) & \cdots & p_{N-1}(\sigma_{1})\\
\vdots &  & \vdots\\
p_{0}(\sigma_{N}) & \cdots & p_{N-1}(\sigma_{N})
\end{bmatrix}=:V,\label{eq:V}
\end{align}
and 
\begin{align}
\bar{V}_{j}' & =\begin{bmatrix}\bar{p}'_{j0}(\bar{\tau}_{j1}) & \cdots & \bar{p}'_{j\,N}(\bar{\tau}_{j1})\\
\vdots &  & \vdots\\
\bar{p}'_{j0}(\bar{\tau}_{j,N+1}) & \cdots & \bar{p}'_{jN}(\bar{\tau}_{j,N+1})
\end{bmatrix}=\begin{bmatrix}0 & p_{0}(\bar{\sigma}_{1}) & \cdots & p_{N-1}(\bar{\sigma}_{1})\\
\vdots & \vdots &  & \vdots\\
0 & p_{0}(\bar{\sigma}_{N+1}) & \cdots & p_{N-1}(\bar{\sigma}_{N+1})
\end{bmatrix}=:\mathring{V},\label{eq:V0}
\end{align}
we represent, for $\kappa=1,\ldots,k$, 
\begin{align*}
X_{j\kappa}:=\begin{bmatrix}x_{j\kappa}(\bar{\tau}_{j1})\\
\vdots\\
x_{j\kappa}(\bar{\tau}_{j,N+1})
\end{bmatrix}=\bar{V}_{j}c_{j\kappa}=h_{j}\bar{V}c_{j\kappa},\\
X'_{j\kappa}:=\begin{bmatrix}x'_{j\kappa}(\bar{\tau}_{j1})\\
\vdots\\
x'_{j\kappa}(\bar{\tau}_{j,N+1})
\end{bmatrix}=\bar{V}'_{j}c_{j\kappa}=\mathring{V}c_{j\kappa},
\end{align*}
and, for $\kappa=k+1,\ldots,m$, 
\begin{align*}
X_{j\kappa}:=\begin{bmatrix}x_{j\kappa}(\tau_{j1})\\
\vdots\\
x_{j\kappa}(\tau_{jN})
\end{bmatrix}=V_{j}c_{j\kappa}=Vc_{j\kappa}.
\end{align*}
The matrices $\bar{V}$ and $V$ are nonsingular. This amounts to
the relation 
\begin{align}
c_{j}=\begin{bmatrix}c_{j1}\\
\vdots\\
c_{jk}\\
c_{j,k+1}\\
\vdots\\
c_{jm}
\end{bmatrix}=\begin{bmatrix}I_{k}\otimes\bar{V}^{-1}\\
 & I_{m-k}\otimes V^{-1}
\end{bmatrix}\begin{bmatrix}\frac{1}{h_{j}}X_{j1}\\
\vdots\\
\frac{1}{h_{j}}X_{jk}\\
X_{j,k+1}\\
\vdots\\
X_{jm}
\end{bmatrix},\;j=1,\ldots,n.\label{Inverse}
\end{align}
Owing to the fact, that polynomials of degree $N$ and $N-1$ are
uniquely determined by their values at $N+1$ and $N$ different nodes,
respectively, formula (\ref{Inverse}) provides $c=\mathcal{R}^{-1}x$
for each arbitrary given $x\in\tilde{X}_{\pi}$.\medskip{}

Next, we equip $\tilde{X}_{\pi}$ with the norms 
\begin{align}
\lVert x\rVert_{L^{2}}^{2} & =\sum_{j=1}^{n}\left\{ \sum_{\kappa=1}^{k}\int_{t_{j-1}}^{t_{j}}\lvert x_{j\kappa}(t)\rvert^{2}\dt t+\,\sum_{\kappa=k+1}^{m}\int_{t_{j-1}}^{t_{j}}\lvert x_{j\kappa}(t)\rvert^{2}\dt t\,\right\} ,\label{eq:L-1}\\
\lVert x\rVert_{H_{D,\pi}^{1}}^{2} & =\sum_{j=1}^{n}\left\{ \sum_{\kappa=1}^{k}\int_{t_{j-1}}^{t_{j}}(\lvert x_{j\kappa}(t)\rvert^{2}+\lvert x_{j\kappa}'(t)\rvert^{2})\dt t+\,\sum_{\kappa=k+1}^{m}\int_{t_{j-1}}^{t_{j}}\lvert x_{j\kappa}(t)\rvert^{2}\dt t\,\right\} .\label{eq:H}
\end{align}
The latter norm reduces, for $x\in X_{\pi}$, to $\lVert x\rVert_{H_{D,\pi}^{1}}=\|x\|_{H_{D}^{1}(a,b)}$.
Moreover, $\lVert\cdot\rVert_{L^{2}}=\lVert\cdot\rVert_{L^{2}((a,b),\R^{m})}$.
On $\R^{n(mN+k)}$, we use the Euclidean norm. Then $\mathcal{R}$
becomes a homeomorphism in each case, and we are interested in the
respective operator norms $\lVert\mathcal{R}\rVert_{\R^{n(mN+k)}\rightarrow L^{2}}$,
$\lVert\mathcal{R}\rVert_{\R^{n(mN+k)}\rightarrow H_{D,\pi}^{1}}$,
$\|\mathcal{R}^{-1}\|_{L^{2}\rightarrow\R^{n(mN+k)}}$, and $\|\mathcal{R}^{-1}\|_{H_{D,\pi}^{1}\rightarrow\R^{n(mN+k)}}$.
Regarding the properties of the related integration formulae and introducing
the diagonal matrices 
\begin{equation}
\bar{\Gamma}=\diag(\bar{\gamma}_{1}^{1/2},\cdots,\bar{\gamma}_{N+1}^{1/2}),\;\Gamma=\diag(\gamma_{1}^{1/2},\cdots,\gamma_{N}^{1/2})\label{eq:Gam}
\end{equation}
we compute for any $x=\mathcal{R}c$, and $\kappa=1,\ldots,k$, 
\begin{align*}
\int_{t_{j-1}}^{t_{j}}\lvert x_{j\kappa}(t)\rvert^{2}\dt t & =h_{j}\sum_{i=1}^{N+1}\bar{\gamma}_{i}\lvert x_{j\kappa}(\bar{\tau}_{ji})\rvert^{2}=h_{j}\sum_{i=1}^{N+1}\lvert\bar{\gamma}_{i}^{1/2}x_{j\kappa}(\bar{\tau}_{ji})\rvert^{2}=h_{j}\lvert\bar{\Gamma}X_{j\kappa}\rvert^{2}\\
 & =h_{j}\lvert\bar{\Gamma}\bar{V}_{j}c_{j\kappa}\rvert^{2}=h_{j}\lvert\bar{\Gamma}h_{j}\bar{V}c_{j\kappa}\rvert^{2},
\end{align*}
\begin{align*}
\int_{t_{j-1}}^{t_{j}} & (\lvert x_{j\kappa}(t)\rvert^{2}+\lvert x_{j\kappa}'(t)\rvert^{2})\dt t=h_{j}\sum_{i=1}^{N+1}\bar{\gamma}_{i}(\lvert x_{j\kappa}(\bar{\tau}_{ji})\rvert^{2}+\lvert x_{j\kappa}'(\bar{\tau}_{ji})\rvert^{2})\\
 & =h_{j}\sum_{i=1}^{N+1}(\lvert\bar{\gamma}_{i}^{1/2}x_{j\kappa}(\bar{\tau}_{ji})\rvert^{2}+(\lvert\bar{\gamma}_{i}^{1/2}x_{j\kappa}'(\bar{\tau}_{ji})\rvert^{2})=h_{j}\lvert\bar{\Gamma}X_{j\kappa}\rvert^{2}+h_{j}\lvert\bar{\Gamma}X_{j\kappa}'\rvert^{2}\\
 & =h_{j}\lvert\bar{\Gamma}\bar{V}_{j}c_{j\kappa}\rvert^{2}+h_{j}\lvert\bar{\Gamma}\mathring{V}c_{j\kappa}\rvert^{2}=h_{j}\lvert\bar{\Gamma}h_{j}\bar{V}c_{j\kappa}\rvert^{2}+h_{j}\lvert\bar{\Gamma}\mathring{V}c_{j\kappa}\rvert^{2}\\
 & =h_{j}\left|\begin{bmatrix}h_{j}\bar{\Gamma}\bar{V}\\
\bar{\Gamma}\mathring{V}
\end{bmatrix}c_{j\kappa}\right|^{2},
\end{align*}
and, in addition, for $\kappa=k+1,\ldots,m$, 
\begin{align*}
\int_{t_{j-1}}^{t_{j}}\lvert x_{j\kappa}(t)\rvert^{2}\dt t & =h_{j}\sum_{i=1}^{N}\gamma_{i}\lvert x_{j\kappa}(\tau_{ji})\rvert^{2}=h_{j}\sum_{i=1}^{N}\lvert\gamma_{i}^{1/2}x_{j\kappa}(\tau_{ji})\rvert^{2}=h_{j}\lvert\Gamma X_{j\kappa}\rvert^{2}\\
 & =h_{j}\lvert\Gamma Vc_{j\kappa}\rvert^{2}.
\end{align*}
Summarizing, the following representations result: 
\begin{align}
\rVert x\lVert_{L^{2}}^{2}=\sum_{j=1}^{n}\left\{ \sum_{\kappa=1}^{k}\lvert h_{j}^{3/2}\bar{\Gamma}\bar{V}c_{j\kappa}\rvert^{2}+\sum_{\kappa=k+1}^{m}\lvert h_{j}^{1/2}\Gamma Vc_{j\kappa}\rvert^{2}\right\} =\sum_{j=1}^{n}\lvert U_{j}c_{j}\rvert^{2}=|\mathcal{U}c|^{2},\label{WTW}
\end{align}
with matrices 
\begin{align}
\mathcal{U} & =\diag(U_{1},\cdots,U_{n})\in\R^{n(mN+k)\times n(mN+k)},\label{eq:U}\\
U_{j} & =\begin{bmatrix}I_{k}\otimes h_{j}^{3/2}\bar{\Gamma}\bar{V}\\
 & I_{m-k}\otimes h_{j}^{1/2}\Gamma V
\end{bmatrix}\in\R^{(mN+k)\times(mN+k)},\nonumber 
\end{align}
and 
\begin{align}
\rVert x\lVert_{H_{D,\pi}^{1}}^{2}=\sum_{j=1}^{n}\left\{ \sum_{\kappa=1}^{k}\left|\begin{bmatrix}h_{j}^{3/2}\bar{\Gamma}\bar{V}\\
h_{j}^{1/2}\bar{\Gamma}\mathring{V}
\end{bmatrix}c_{j\kappa}\right|^{2}+\sum_{\kappa=k+1}^{m}\lvert h_{j}^{1/2}\Gamma Vc_{j\kappa}\rvert^{2}\right\} =\sum_{j=1}^{n}\lvert\hat{U}_{j}c_{j}\rvert^{2}=|\hat{\mathcal{U}}c|^{2},\label{eq:WTWh}
\end{align}
with matrices 
\begin{align}
\hat{\mathcal{U}} & =\diag(\hat{U}_{1},\cdots,\hat{U}_{n})\in\R^{n(mN+k+k(N+1))\times n(mN+k)},\label{eq:Uhat}\\
\hat{U}_{j} & =\begin{bmatrix}I_{k}\otimes\begin{bmatrix}h_{j}^{3/2}\bar{\Gamma}\bar{V}\\
h_{j}^{1/2}\bar{\Gamma}\mathring{V}
\end{bmatrix}\\
 & I_{m-k}\otimes h_{j}^{1/2}\Gamma V
\end{bmatrix}\in\R^{(mN+k+k(N+1))\times(mN+k)}.\nonumber 
\end{align}

\begin{prop}
\label{prop:The-singular-values}The singular values of $\mathcal{U}$
and $\hat{\mathcal{U}}$ are independent of the choice of the nodes
$\sigma_{i}$ and $\bar{\sigma}_{i}$. Moreover, all singular values
are positive.
\end{prop}
\begin{proof}
$U_{j}$ and $\hat{U}_{j}$ have full column-rank. Consequently $\mathcal{U}^{T}\mathcal{U}$
and $\hat{\mathcal{U}}^{T}\hat{\mathcal{U}}$ are symmetric and positive
definite. Hence, their eigenvalues are all positive and, thus, also
their singular values being the square root of the eigenvalues. The
eigenvalues are independent of the choice of the nodes $\sigma_{i}$
and $\bar{\sigma}_{i}$ since, owing to the properties of the involved
integration formulae, it holds that
\begin{align*}
(V^{T}\Gamma^{2}V)_{\alpha\beta} & =\int_{0}^{1}p_{\alpha-1}p_{\beta-1}(\sigma){\rm d\sigma},\;\alpha,\beta=1,\cdots,N,\\
(\bar{V}^{T}\bar{\Gamma}^{2}\bar{V})_{\alpha\beta} & =\int_{0}^{1}\bar{p}_{\alpha-1}\bar{p}_{\beta-1}(\sigma){\rm d\sigma},\;\alpha,\beta=1,\cdots,N+1,\\
(\mathring{V}^{T}\bar{\Gamma}^{2}\mathring{V})_{\alpha\beta} & =\int_{0}^{1}\bar{p}'_{\alpha-1}\bar{p}'_{\beta-1}(\sigma){\rm d\sigma},\;\alpha,\beta=1,\cdots,N+1,
\end{align*}
such that the entries of $\mathcal{U}^{T}\mathcal{U}$ and $\hat{\mathcal{U}}^{T}\hat{\mathcal{U}}$
are independent of the choice of the integration formulae.
\end{proof}
\begin{thm}
\label{th:Rnorm}Let $\sigma_{\textrm{min}}(\mathcal{U})$ and $\sigma_{\textrm{max}}(\mathcal{U})$
denote the maximal and minimal singular values of $\mathcal{U}$.
Similarly, let $\sigma_{\textrm{min}}(\mathcal{\hat{U}})$ and $\sigma_{\textrm{max}}(\mathcal{\hat{U}})$
denote the maximal and minimal singular values of $\hat{\mathcal{U}}$.
Then it holds
\begin{align*}
\lVert\mathcal{R}\rVert_{\R^{n(mN+k)}\rightarrow L^{2}} & =\sigma_{\textrm{max}}(\mathcal{U}),\quad\|\mathcal{R}^{-1}\|_{L^{2}\rightarrow\R^{n(mN+k)}}=\sigma_{\textrm{min}}(\mathcal{U})^{-1},\\
\lVert\mathcal{R}\rVert_{\R^{n(mN+k)}\rightarrow H_{D,\pi}^{1}} & =\sigma_{\textrm{max}}(\mathcal{\hat{U}}),\quad\|\mathcal{R}^{-1}\|_{H_{D,\pi}^{1}\rightarrow\R^{n(mN+k)}}=\sigma_{\textrm{min}}(\mathcal{\hat{U}})^{-1}.
\end{align*}
\end{thm}
\begin{proof}
It holds $\mathcal{\hat{U}}\in\R^{\nu\times\lambda}$ with $\nu=n(mN+k+k(N+1))$
and $\lambda=n(mN+k)$. Let $\hat{\mathcal{U}}=U\Sigma V^{T}$ be
the singular value decomposition of $\mathcal{U}$. Here,
\[
\Sigma=\left[\begin{array}{ccc}
s_{1}\\
 & \ddots\\
 &  & s_{\nu}\\
0 & \cdots & 0
\end{array}\right]\in\R^{\nu\times\lambda}
\]
with $s_{1}=\sigma_{\textrm{max}}(\hat{\mathcal{U}})$ and $s_{\nu}=\sigma_{\textrm{min}}(\hat{\mathcal{U}})$.
According to Proposition~\ref{prop:The-singular-values}, $\sigma_{\textrm{min}}(\hat{\mathcal{U}})>0$.
By (\ref{eq:WTWh}), this leads to
\[
\lVert\mathcal{R}\rVert_{\R^{n(mN+k)}\rightarrow H_{D,\pi}^{1}}=\sup_{c\neq0}\frac{\lVert\mathcal{R}c\rVert_{H_{D,\pi}^{1}}}{\lvert c\rvert_{\R^{\lambda}}}=\sup_{c\neq0}\frac{\lvert\hat{\mathcal{U}}c\rvert_{\R^{\nu}}}{\lvert c\rvert_{\R^{\lambda}}}=\sup_{\chi\neq0}\frac{\lvert\Sigma\chi\rvert_{\R^{\nu}}}{\lvert\chi\rvert_{\R^{\lambda}}}=\sigma_{\textrm{max}}(\hat{\mathcal{U}})
\]
and
\[
\lVert\mathcal{R}^{-1}\rVert_{H_{D,\pi}^{1}\rightarrow\R^{n(mN+k)}}=\sup_{x\neq0}\frac{\lvert\mathcal{R}^{-1}x\rvert_{\R^{\lambda}}}{\lVert x\rVert_{H_{D,\pi}^{1}}}=\sup_{c\neq0}\frac{\lvert c\rvert_{\R^{\lambda}}}{\lVert\mathcal{R}c\rVert_{H_{D,\pi}^{1}}}=\sup_{\chi\neq0}\frac{\lvert\chi\rvert_{\R^{\lambda}}}{\lvert\Sigma\chi\rvert_{\R^{\nu}}}=\sigma_{\textrm{min}}(\hat{\mathcal{U}})^{-1}.
\]
The statements concerning $\lVert\mathcal{R}\rVert_{\R^{n(mN+k)}\rightarrow L^{2}}$
and $\|\mathcal{R}^{-1}\|_{H_{D,\pi}^{1}\rightarrow\R^{n(mN+k)}}$
follow similarly.
\end{proof}
Using the structure (\ref{eq:U}) of $\mathcal{U}$, we obtain
\begin{align*}
\sigma_{\textrm{max}}(\mathcal{U}) & =\max_{j=1,\ldots,n}\max\{h_{j}^{3/2}\sigma_{\textrm{max}}(\bar{\Gamma}\bar{V}),h_{j}^{1/2}\sigma_{\textrm{max}}(\Gamma V)\}\\
 & =\max_{j=1,\ldots,n}h_{j}^{1/2}\max\{h_{j}\sigma_{\textrm{max}}(\bar{\Gamma}\bar{V}),\sigma_{\textrm{max}}(\Gamma V)\},\\
\sigma_{\textrm{min}}(\mathcal{U}) & =\min_{j=1,\ldots,n}\min\{h_{j}^{3/2}\sigma_{\textrm{min}}(\bar{\Gamma}\bar{V}),h_{j}^{1/2}\sigma_{\textrm{min}}(\Gamma V)\}\\
 & =\min_{j=1,\ldots,n}h_{j}^{1/2}\min\{h_{j}\sigma_{\textrm{min}}(\bar{\Gamma}\bar{V}),\sigma_{\textrm{min}}(\Gamma V)\}.
\end{align*}
The estimation of the singular values of $\hat{\mathcal{U}}$ leads
to slightly more involved expressions. Let $U_{j,\textrm{red}}=\left[\begin{array}{c}
h_{j}\bar{\Gamma}\bar{V}\\
\bar{\Gamma}\mathring{V}
\end{array}\right]$. Then, it holds
\begin{align*}
\sigma_{\textrm{max}}(\mathcal{\hat{U}}) & =\max_{j=1,\ldots,n}h_{j}^{1/2}\max\{\sigma_{\textrm{max}}(U_{j,\textrm{red}}),\sigma_{\textrm{max}}(\Gamma V)\},\\
\sigma_{\textrm{min}}(\hat{\mathcal{U}}) & =\min_{j=1,\ldots,n}h_{j}^{1/2}\min\{\sigma_{\textrm{min}}(U_{j,\textrm{red}}),\sigma_{\textrm{min}}(\Gamma V)\}.
\end{align*}
We note that $\sigma_{\min}(\bar{\Gamma}\mathring{V})=0$ and $\sigma_{\max}(\Gamma V)=\sigma_{\max}(\bar{\Gamma}\mathring{V})$.
This follows immediately from the construction of the basis for the
differential components (\ref{eq:diffBasis}). The definition of singular
values and Weyl's Theorem \cite[Theorem III.2.1]{Bhatia97} provides
us with
\begin{align*}
\lambda_{\textrm{max}}(\mathring{V}^{T}\bar{\Gamma}^{2}\mathring{V})\leq & \lambda_{\textrm{max}}(h_{j}^{2}\bar{V}^{T}\bar{\Gamma}^{2}\bar{V}+\mathring{V}^{T}\bar{\Gamma}^{2}\mathring{V})=\sigma_{\textrm{max}}(U_{j,\textrm{red}})^{2}\\
 & \leq h_{j}^{2}\lambda_{\textrm{max}}(\bar{V}^{T}\bar{\Gamma}^{2}\bar{V})+\lambda_{\textrm{max}}(\mathring{V}^{T}\bar{\Gamma}^{2}\mathring{V}),\\
h_{j}^{2}\lambda_{\textrm{min}}(\bar{V}^{T}\bar{\Gamma}^{2}\bar{V})\leq & \lambda_{\textrm{min}}(h_{j}^{2}\bar{V}^{T}\bar{\Gamma}^{2}\bar{V}+\mathring{V}^{T}\bar{\Gamma}^{2}\mathring{V})=\sigma_{\textrm{min}}(U_{j,\textrm{red}})^{2}\leq h_{j}^{2}\lambda_{\textrm{max}}(\bar{V}^{T}\bar{\Gamma}^{2}\bar{V})
\end{align*}
since $\lambda_{\min}(\mathring{V}^{T}\bar{\Gamma}^{2}\mathring{V})=0$.
Then,
\begin{align*}
\sigma_{\max}(\Gamma V)= & \lambda_{\max}(V^{T}\Gamma^{2}V)^{1/2}\\
\leq & \max\{\sigma_{\textrm{max}}(U_{j,\textrm{red}}),\sigma_{\textrm{max}}(\Gamma V)\}\\
\leq & \sigma_{\textrm{max}}(\Gamma V)+O(h).
\end{align*}
Moreover, 
\begin{align*}
\min\{h_{j}\sigma_{\min}(\bar{\Gamma}\bar{V}),\sigma_{\min}(\Gamma V)\} & \leq\min\{\sigma_{\textrm{min}}(U_{j,\textrm{red}}),\sigma_{\min}(\Gamma V)\}\\
 & \leq\min\{h_{j}\sigma_{\max}(\bar{\Gamma}\bar{V}),\sigma_{\min}(\Gamma V)\}
\end{align*}

Collecting all estimates Theorem~\ref{th:Rnorm} provides 
\begin{thm}
\label{thm:asymptotic}Let the grid (\ref{eq:mesh}) be equidistant
with the stepsize $h$. Furthermore, let $\Gamma$ and $\bar{\Gamma}$
be given by (\ref{eq:Gam}) and let $V$, $\bar{V}$ and $\mathring{V}$
be given by (\ref{eq:V}), (\ref{eq:Vb}), (\ref{eq:V0}). Then it
holds, for sufficiently small $h$,

\begin{align*}
\lVert\mathcal{R}\rVert_{\R^{n(mN+k)}\rightarrow L^{2}} & =h^{1/2}\sigma_{\max}(\Gamma V)=O(h^{1/2}),\\
\|\mathcal{R}^{-1}\|_{L^{2}\rightarrow\R^{n(mN+k)}} & =h^{-3/2}\sigma_{\min}(\bar{\Gamma}\bar{V})^{-1}=O(h^{-3/2}),\\
\lVert\mathcal{R}\rVert_{\R^{n(mN+k)}\rightarrow H_{D,\pi}^{1}} & =h^{1/2}\sigma_{\max}(\Gamma V)+O(h^{3/2})=O(h^{1/2}),
\end{align*}
and
\[
h^{-3/2}\sigma_{\max}(\bar{\Gamma}\bar{V})^{-1}\leq\lVert\mathcal{R}^{-1}\rVert_{H_{D,\pi}^{1}\rightarrow\R^{n(mN+k)}}\leq h^{-3/2}\sigma_{\min}(\bar{\Gamma}\bar{V})^{-1}.
\]
In particular, $\lVert\mathcal{R}\rVert_{\R^{n(mN+k)}\rightarrow H_{D,\pi}^{1}}=\lVert\mathcal{R}\rVert_{\R^{n(mN+k)}\rightarrow L^{2}}+O(h^{3/2})$.
\end{thm}
In these estimates we used the fact $\sigma_{\min}(\Gamma V)>0$.
Note that the constants hidden in the big-O notation in this theorem
depend both on $N$ and the chosen basis. For the restriction $\tilde{\mathcal{R}}$
of $\mathcal{R}$ onto $\ker\mathcal{C}$ we obtain, obviously,
\begin{align*}
\lVert\tilde{\mathcal{R}}\rVert & \leq\lVert\mathcal{R}\rVert,\quad\lVert\tilde{\mathcal{R}}^{-1}\rVert\leq\lVert\mathcal{R}^{-1}\rVert.
\end{align*}

For some special cases, the singular values can be easily derived.
\begin{prop}
\label{RemSig}Let $V$, $\bar{V}$, and $\mathring{V}$ be given
by (\ref{eq:Vb})-(\ref{eq:V0}) and $\Gamma$, $\bar{\Gamma}$ by
(\ref{eq:Gam}). Then it holds:
\begin{description}
\item [{{\rm (1)}}] Let $p_{0},\ldots,p_{N-1}$ be an orthogonal basis
in $L^{2}(0,1)$. Then
\begin{align*}
\sigma_{\min}(\Gamma V) & =\min\left\{ \lVert p_{\alpha}\|_{L^{2}(0,1)}:\alpha=0,\ldots,N-1\right\} ,\\
\sigma_{\max}(\Gamma V) & =\max\left\{ \lVert p_{\alpha}\|_{L^{2}(0,1)}:\alpha=0,\ldots,N-1\right\} .
\end{align*}
In particular, if $p_{0},\ldots,p_{N-1}$ is the Legendre basis, $\sigma_{\min}(\Gamma V)=(2N-1)^{-1/2}$
and $\sigma_{\max}(\Gamma V)=1$.
\item [{{\rm (2)}}] For an orthonormal basis $p_{0},\ldots,p_{N-1}$ in
$L^{2}(0,1)$, $\sigma_{\min}(\Gamma V)=\sigma_{\max}(\Gamma V)=1$.
\item [{{\rm (3)}}] If $p_{0},\ldots,p_{N-1}$ is the modified Legendre
basis, it holds $\sigma_{\min}(\bar{\Gamma}\bar{V})\geq(2N+1)^{-1/2}$
and $\sigma_{\max}(\bar{\Gamma}\bar{V})\leq(N+2)^{1/2}$. Furthermore,
the estimates
\[
\sigma_{\min}(\Gamma V)\geq\left(\frac{1}{2-2\cos\frac{N}{N+2}\pi}\right)^{1/2}\geq\frac{1}{2},\quad\sigma_{\max}(\Gamma V)\leq\left(\frac{2N-1}{2-2\cos\frac{1}{N+2}\pi}\right)^{1/2}
\]
hold true.
\end{description}
\end{prop}
\begin{proof}
First, we observe that $(V^{T}\Gamma^{2}V)_{\alpha\beta}=\int_{0}^{1}p_{\alpha-1}(\rho)p_{\beta-1}(\rho)\dt\rho=\delta_{\alpha\beta}\lVert p_{\alpha-1}\rVert_{L^{2}(a,b)}^{2}$.
This provides (1) and (2) as special cases. 

Consider the modified Legendre basis now. It holds $\int_{0}^{1}\bar{p}_{0}^{2}(\rho)\dt\rho=1$
and $\int_{0}^{1}\bar{p}_{0}(\rho)\bar{p}_{\alpha}(\rho)\dt\rho=\int_{0}^{1}(P_{\alpha}(2\rho-1)-(-1)^{\alpha})\dt\rho=(-1)^{\alpha+1}$
for $\alpha=1,2,\ldots$. Moreover, for $\alpha,\beta=1,2,\ldots,$
we have
\begin{align*}
\int_{0}^{1}\bar{p}_{\alpha}(\rho)\bar{p}_{\beta}(\rho)\dt\rho & =\int_{0}^{1}(P_{\alpha}(2\rho-1)-(-1)^{\alpha})(P_{\beta}(2\rho-1)-(-1)^{\beta})\dt\rho\\
 & =\int_{0}^{1}P_{\alpha}(2\rho-1)P_{\beta}(2\rho-1)\dt\rho+(-1)^{\alpha+\beta}\\
 & =(2\alpha+1)^{-1}\delta_{\alpha\beta}+(-1)^{\alpha+\beta}.
\end{align*}
Collecting these expressions, we obtain the compact representation
\[
\bar{V}^{T}\bar{\Gamma}^{2}\bar{V}=\diag(1,\frac{1}{3},\ldots,(2N+1)^{-1})+ff^{T}
\]
 with $f^{T}=[1,1,-1,+1,-1,\ldots,\pm1]\in\R^{N+1}$. $ff^{T}$ is
a rank-1 matrix having, therefore, the $N$-fold eigenvalue 0. Moreover,
$f$ is an eigenvector to the eigenvalue $f^{T}f=N+1$. In particular,
$ff^{T}$ is positive semidefinite. Invoking Weyl's theorem again,
we obtain 
\begin{align*}
(2N+1)^{-1} & =\lambda_{\min}(\diag(1,\frac{1}{3},\ldots,(2N+1)^{-1}))\leq\lambda_{\min}(\bar{V}^{T}\bar{\Gamma}^{2}\bar{V})\\
\lambda_{\max}(\bar{V}^{T}\bar{\Gamma}^{2}\bar{V}) & \leq\lambda_{\max}(\diag(1,\frac{1}{3},\ldots,(2N+1)^{-1}))+\lambda_{\max}(ff^{T})=N+2.
\end{align*}
This proves the first assertion of (3).

The relation $(\mathring{V}^{T}\bar{\Gamma}^{2}\mathring{V})_{\alpha\beta}=\int_{0}^{1}\bar{p}'_{\alpha-1}\bar{p}'_{\beta-1}(\sigma){\rm d\sigma}$
shows that $K=\mathring{V}^{T}\bar{\Gamma}^{2}\mathring{V}$ is the
stiffness matrix of the basis functions. For the modified Legendre
basis, it has been investigated in \cite[cp Eq. (31)]{HMTWW}. According
to the proof of Proposition A.2 of \cite{HMTWW}, the nonvanishing
eigenvalues can be estimated by\footnote{In \cite{HMTWW}, the stiffness matrix is scaled on the interval $(-1,1)$
in contrast to the interval $(0,1)$ used here. Therefore, an additional
factor of $1/2$ appears in the present estimations.}
\[
\lambda_{\min}(K)\geq\frac{1}{2-2\cos\frac{N}{N+2}\pi},\quad\lambda_{\max}(K)\leq\frac{2N-1}{2-2\cos\frac{1}{N+2}\pi}.
\]
$K'=V^{T}\Gamma^{2}V$ is the submatrix of $K$ obtained by omitting
the first row and column of $K$, which consist entirely of zeros.
This provides the final relations of assertion (3).
\end{proof}
An asymptotic analysis shows that $\sigma_{\max}(\Gamma V)\leq\frac{2}{\sqrt{\pi}}N^{3/2}+O(N^{1/2})$
in the case of the modified Legendre basis.
\begin{rem}
We are able to estimate the size of the jump of elements of $\tilde{X}_{\pi}$
at the grid points. For any $\tilde{x}\in\tilde{X}_{\pi}$ and $\tilde{c}=\mathcal{R}^{-1}\tilde{x}$,
it holds
\[
\lVert\tilde{x}_{j\kappa}\lVert_{C[t_{j-1},t_{j})}\leq C_{h_{j}}\lVert\tilde{x}_{j\kappa}\rVert_{H^{1}(t_{j-1},t_{j})}=C_{h_{j}}h_{j}^{1/2}\lvert U_{j,\textrm{red}}\tilde{c}_{j\kappa}\rvert\leq C_{h_{j}}h_{j}^{1/2}\sigma_{\max}(U_{j,\textrm{red}})\lvert\tilde{c}\rvert
\]
with $C_{h_{j}}=\left(\max\{2/h_{j},h_{j}\}\right)^{1/2}$. Here,
we used \cite[Lemma 3.2]{HMinival}. For sufficiently small $h_{j}$,
this estimate reduces to
\[
\lVert\tilde{x}_{j\kappa}\lVert_{C[t_{j-1},t_{j})}\leq\sqrt{2}\sigma_{\max}(\bar{\Gamma}\mathring{V})\lvert\tilde{c}\rvert=\sqrt{2}\sigma_{\max}(\Gamma V)\lvert\tilde{c}\rvert.
\]
Let $x$ be any element of $X_{\pi}$ and $c=\mathcal{R}^{-1}x$.
Replacing $\tilde{c}$ by $\Delta c=\tilde{c}-c$ in the last estimate
we obtain
\begin{align*}
\lvert\tilde{x}_{\kappa}(t_{j-0})-\tilde{x}_{\kappa}(t_{j+0})\rvert & =\lvert\tilde{x}_{\kappa}(t_{j-0})-x_{\kappa}(t_{j-0})+x_{\kappa}(t_{j+0})-\tilde{x}_{\kappa}(t_{j+0})\rvert\\
 & \leq\lVert\tilde{x}_{\kappa j}-x_{\kappa j}\rVert_{C[t_{j-1},t_{j})}+\lVert\tilde{x}_{\kappa,j+1}-x_{\pi,\kappa,j+1}\rVert_{C[t_{j},t_{j+1})}\\
 & \leq2\sqrt{2}\sigma_{\max}(\Gamma V)\lvert\Delta c\rvert.
\end{align*}
Proposition~\ref{RemSig} provides estimations for the factor $\sigma_{\max}(\Gamma V)$.
In particular, for some bases, it does not depend on the polynomial
degree $N$.\qed
\end{rem}

\section{Error estimation for the constrained minimization problem}

The aim of this section is the derivation of bounds for perturbations
of the solution $c$ for the problem (\ref{eq:discmin})-(\ref{eq:discconstraint}),
that is,
\[
\begin{gathered}\varphi(c)=\lvert\mathcal{A}c-r\rvert^{2}\rightarrow\min!\\
\textrm{such that }\mathcal{C}c=0,
\end{gathered}
\]
 under perturbation of the data $\mathcal{A}$, $\mathcal{C}$, $r$.
Such bounds are know for a long time, e.g., \cite{Elden80,CoxHig99}.
However, we will provide different bounds in this section. The reason
for this is that the constraint $\mathcal{C}c=0$ has an exceptional
meaning in the present context: It holds $\mathcal{C}c=0$ if and
only if $\mathcal{R}c\in H_{D}^{1}(a,b)$. If a perturbation $\Delta\mathcal{C}$
of $\mathcal{C}$ changes the kernel of $\mathcal{C}$, it does no
longer hold $\mathcal{R}c\in H_{D}^{1}(a,b)$ in general! Therefore,
we will consider the two cases $\ker(\mathcal{C}+\Delta\mathcal{C})=\ker\mathcal{C}$
and $\ker(\mathcal{C}+\Delta\mathcal{C})\neq\ker\mathcal{C}$ separately.

Let $\tilde{c}$ the solution of the perturbed problem
\begin{equation}
\min\{(\mathcal{A}+\Delta\mathcal{A})z-(r+\Delta r)\vert(\mathcal{C}+\Delta\mathcal{C})z=0\}.\label{eq:pertmini}
\end{equation}
Then, let $\Delta c=c-\tilde{c}$ denote the error. We are interested
in deriving an error bound on $\Delta c$ in terms of the perturbations
of the data.

Let, for a matrix $\mathcal{M}$, denote the
Moore-Penrose inverse by $\mathcal{M}^{+}$. Moreover, let $\lVert\mathcal{M}\rVert$ be
its spectral norm.

Let $\mathcal{D}$ be an orthonormal basis of $\ker\mathcal{C}$.
Then, $P=I_{n(mN+k)}-\mathcal{C}^{+}\mathcal{C}$ is the orthogonal
projector onto $\ker\mathcal{C}$ and $P\mathcal{D}=\mathcal{D}$.
Some more properties are collected in the following proposition.
\begin{prop}
\label{prop:DP}It holds, for any matrix $\mathcal{M}\in\R^{\nu\times n(mN+k)},$
$\nu\in\Natu$,
\begin{enumerate}
\item $\mathcal{D}^{T}\mathcal{D}=I_{nmN+k}$ and $\mathcal{D}\mathcal{D}^{T}=P$.
\item If $c=\mathcal{D}d,$ then $\lvert c\rvert=\lvert d\rvert$.
\item $\lVert\mathcal{A}\mathcal{D}\rVert=\lVert\mathcal{A}P\rVert$.
\item $(\mathcal{A}P)^{+}=\mathcal{D}(\mathcal{A}\mathcal{D})^{+}$.
\item $\lVert(\mathcal{A}P)^{+}\rVert=\lVert(\mathcal{A}\mathcal{D})^{+}\rVert$.
\end{enumerate}
\end{prop}
The proofs are obvious. For the following we note that the matrix
$\mathcal{A}\mathcal{D}$ has full column rank \cite[Proposition 1]{HM_Part2}.

\subsection{$\ker(\mathcal{C}+\Delta\mathcal{C})=\ker\mathcal{C}$}

Each element $c$ of $\ker\mathcal{C}$ has a unique representation
$c=\mathcal{D}d$ with $d\in\R^{nmN+k}$. Therefore, (\ref{eq:discmin})-(\ref{eq:discconstraint})
is equivalent to the unconstrained minimization problem
\begin{equation}
\min_{d\in\R^{nmN+k}}\lVert\mathcal{A}\mathcal{D}d-r\rVert\label{eq:proj}
\end{equation}
while (\ref{eq:pertmini}) becomes the unconstrained minimization
problem
\begin{equation}
\min_{d\in\R^{nmN+k}}\lVert(\mathcal{A}+\Delta\mathcal{A})\mathcal{D}d-(r+\Delta r)\rVert.\label{eq:projpert}
\end{equation}
Since $\mathcal{A}\mathcal{D}$ has full column rank, standard perturbation
results for unconstrained least squares problems apply. As a consequence
of \cite[Satz 8.2.7]{KS88} and Proposition~\ref{prop:DP} we obtain
\begin{thm}
\label{Theo:constant-nullspace}Let $\omega=\lVert(\mathcal{A}P)^{+}\rVert\lVert\Delta\mathcal{A}P\rVert<1$.
Then it holds
\[
\lvert\Delta c\rvert\leq\frac{\lVert(\mathcal{A}P)^{+}\rVert}{1-\omega}\left\{ \lVert\Delta\mathcal{A}P\rVert\left[\lvert c\rvert+\lVert(\mathcal{A}P)^{+}\rVert\lvert\mathfrak{r}\rvert+\lvert\Delta r\rvert\right]\right\} 
\]
and
\begin{align*}
\frac{\lvert\Delta c\rvert}{\lvert c\rvert} & \leq\frac{1}{1-\omega}\left\{ \left[\kappa_{\mathcal{C}}(\mathcal{A})+\frac{\lvert\mathfrak{r}\rvert}{\lVert\mathcal{A}P\rVert\lvert c\rvert}\kappa_{\mathcal{C}}(\mathcal{A})^{2}\right]\frac{\lVert\Delta\mathcal{A}P\rVert}{\lVert\mathcal{A}P\rVert}\right.\\
 & \left.+\frac{\lVert(\mathcal{A}P)^{+}\rVert\lvert r\rvert}{\lvert c\rvert}\cdot\frac{\lvert\Delta r\rvert}{\lvert r\rvert}\right\} .
\end{align*}
Here, $\mathfrak{r}=r-\mathcal{A}c$ and 
\[
\kappa_{\mathcal{C}}(\mathcal{A})=\lVert\mathcal{A}P\rVert\lVert(\mathcal{A}P)^{+}\rVert.
\]
\end{thm}
Theorem~\ref{Theo:constant-nullspace} corresponds to classical results
for unconstrained minimization problems (e.g., \cite{We73}, \cite[Theorem 9.12]{LaHa74}
and is a small generalization of them. Let us emphasize that the estimation
is independent of the perturbations of $\mathcal{C}$ as long as the
nullspace of $\mathcal{C}$ is not changed by the perturbation.
\begin{rem}
In the case of the Legendre basis, the elements of $\mathcal{C}$
consist only of three nonzero elements being equal to 1 and -1, respectively,
possibly scaled by the stepsizes, cf.~(\ref{eq:barPj}), (\ref{eq:barPj2}).
So we expect $\Delta\mathcal{C}=0$ such that the estimates of this
section apply.\qed
\end{rem}

\subsection{$\ker(\mathcal{C}+\Delta\mathcal{C})\protect\neq\ker\mathcal{C}$ }

The estimation of the error becomes much more involved than in the
previous case. In a first step, we will construct a basis for the
kernel of the perturbed constraint $(\mathcal{C}+\Delta\mathcal{C})z=0$.
\begin{lem}
\label{Lemma1}Let $\varkappa=\lVert\mathcal{C}^{+}\rVert\lVert\Delta\mathcal{C}\rVert<1/2.$
Then $\mathcal{C}+\Delta\mathcal{C}$ has full rank and $P_{\Delta}=I_{n(mN+k)}-(\mathcal{C}+\Delta\mathcal{C})^{+}(\mathcal{C}+\Delta\mathcal{C})$
is a projector onto $\ker(\mathcal{C}+\Delta\mathcal{C})$. Furthermore,
$\mathcal{D}_{\Delta}=P_{\Delta}\mathcal{D}$ is a basis of $\ker(\mathcal{C}+\Delta\mathcal{C})$.
Moreover, the estimates
\[
\lVert(\mathcal{C}+\Delta\mathcal{C})^{+}\rVert\leq\frac{\lVert\mathcal{C}^{+}\rVert}{1-\varkappa}
\]
and
\[
\lVert(\mathcal{C}+\Delta\mathcal{C})^{+}-\mathcal{C}^{+}\rVert\leq\frac{\sqrt{2}\lVert\mathcal{C}^{+}\rVert^{2}}{1-\varkappa}\lVert\Delta\mathcal{C}\rVert
\]
hold true.
\end{lem}
\begin{proof}
The proposition of $\mathcal{C}+\Delta\mathcal{C}$ having full rank
as well as the error estimates follow from \cite[Satz 8.2.5]{KS88}.

For showing that $\mathcal{D}_{\Delta}$ is a basis of $\ker(\mathcal{C}+\Delta\mathcal{C})$
consider
\begin{align*}
(I-P_{\Delta})P & =(\mathcal{C}+\Delta\mathcal{C})^{+}(\mathcal{C}+\Delta\mathcal{C})(I-\mathcal{C}^{+}\mathcal{C})\\
 & =(\mathcal{C}+\Delta\mathcal{C})^{+}\Delta\mathcal{C}(I-\mathcal{C}^{+}\mathcal{C}).
\end{align*}
It holds
\[
\lVert(I-P_{\Delta})P\rVert\leq\lVert(\mathcal{C}+\Delta\mathcal{C})^{+}\rVert\lVert\Delta\mathcal{C}\rVert\leq\frac{\lVert\mathcal{C}^{+}\rVert}{1-\varkappa}\lVert\Delta\mathcal{C}\rVert\leq\frac{\varkappa}{1-\varkappa}<1.
\]
Therefore, the assumptions of \cite[Theorem I-6.34]{Kato95} are fulfilled.
Since $\dim\ker(\mathcal{C}+\Delta\mathcal{C})=\dim\ker\mathcal{C},$the
first alternative of that theorem applies and $P_{\Delta}$ is a one-to-one
mapping of $\ker\mathcal{C}$ onto $\ker(\mathcal{C}+\Delta\mathcal{C})$.
Hence, $\mathcal{D}_{\Delta}$ is a basis of the latter space.
\end{proof}
By using the bases $\mathcal{D}$ and $\mathcal{D}_{\Delta}$, the
unperturbed and the perturbed least squares problems become (\ref{eq:proj})
and
\begin{equation}
\min_{d\in\R^{nmN+k}}\lVert(\mathcal{A}+\Delta\mathcal{A})\mathcal{D}_{\Delta}d-(r+\Delta r)\rVert.\label{eq:pertbas}
\end{equation}
In a first step, the deviations of the bases shall be estimated. It
holds
\begin{align*}
P_{\Delta}-P & =\mathcal{C}^{+}\mathcal{C}-(\mathcal{C}+\Delta\mathcal{C})^{+}(\mathcal{C}+\Delta\mathcal{C})\\
 & =\mathcal{C}^{+}\mathcal{C}-(\mathcal{C}+\Delta\mathcal{C})^{+}\mathcal{C}-(\mathcal{C}+\Delta\mathcal{C})^{+}\Delta\mathcal{C}\\
 & =\left[\mathcal{C}^{+}-(\mathcal{C}+\Delta\mathcal{C})^{+}\right]\mathcal{C}-(\mathcal{C}+\Delta\mathcal{C})^{+}\Delta\mathcal{C}.
\end{align*}
Invoking Lemma~\ref{Lemma1} we obtain\footnote{In case that $\ker(\mathcal{C}+\Delta\mathcal{C})=\ker\mathcal{C}$
we obtain $\mathcal{P}_{\Delta}-\mathcal{P}=0$ and $\mathcal{D}_{\Delta}=\mathcal{D}$
such that the present estimations coincide with those of the previous
section.}
\[
\lVert P_{\Delta}-P\rVert\leq\left[\frac{\sqrt{2}\lVert\mathcal{C}^{+}\rVert^{2}}{1-\varkappa}\lVert\mathcal{C}\rVert+\frac{\lVert\mathcal{C}^{+}\rVert}{1-\varkappa}\right]\lVert\Delta\mathcal{C}\rVert=\frac{\lVert\mathcal{C}^{+}\rVert}{1-\varkappa}\left[\sqrt{2}\kappa(\mathcal{C})+1\right]\lVert\Delta\mathcal{C}\rVert
\]
with $\kappa(\mathcal{C})=\lVert\mathcal{C}^{+}\rVert\lVert\mathcal{C}\rVert$.
Consequently, 
\begin{equation}
\lVert\mathcal{D}_{\Delta}-\mathcal{D}\rVert=\lVert(P_{\Delta}-P)\mathcal{D}\rVert\leq\lVert P_{\Delta}-P\rVert\lVert\mathcal{D}\rVert\leq\frac{\lVert\mathcal{C}^{+}\rVert}{1-\varkappa}\left[\sqrt{2}\kappa(\mathcal{C})+1\right]\lVert\Delta\mathcal{C}\rVert.\label{eq:baspert}
\end{equation}
Let us transform (\ref{eq:pertbas}) now. It holds
\begin{align*}
(\mathcal{A}+\Delta\mathcal{A})\mathcal{D}_{\Delta} & =(\mathcal{A}+\Delta\mathcal{A})\mathcal{D}+(\mathcal{A}+\Delta\mathcal{A})(\mathcal{D}_{\Delta}-\mathcal{D})\\
 & =\mathcal{A}\mathcal{D}+\mathfrak{R}
\end{align*}
where $\mathfrak{R}=\Delta\mathcal{A}\mathcal{D}+(\mathcal{A}+\Delta\mathcal{A})(\mathcal{D}_{\Delta}-\mathcal{D})$.
The representation of $\mathfrak{R}$ provides the estimate
\begin{equation}
\lVert\mathfrak{R}\rVert\leq\lVert\Delta\mathcal{A}P\rVert+\lVert\mathcal{A}+\Delta\mathcal{A}\rVert\frac{\lVert\mathcal{C}^{+}\rVert}{1-\varkappa}\left[\sqrt{2}\kappa(\mathcal{C})+1\right]\lVert\Delta\mathcal{C}\rVert.\label{eq:Rest}
\end{equation}
Denote $\omega_{\Delta}=\lVert(\mathcal{A}P)^{+}\rVert\lVert\mathfrak{R}\rVert$.
The condition $\omega_{\Delta}<1$ is obviously fulfilled if 
\begin{equation}
\lVert(\mathcal{A}P)^{+}\rVert\left\{ \lVert\Delta\mathcal{A}P\rVert+\lVert\mathcal{A}+\Delta\mathcal{A}\rVert\frac{\lVert\mathcal{C}^{+}\rVert}{1-\varkappa}\left[\sqrt{2}\kappa(\mathcal{C})+1\right]\lVert\Delta\mathcal{C}\rVert\right\} <1.\label{eq:omegaD}
\end{equation}
Let $d+\Delta d$ be the solution of (\ref{eq:pertbas}). Using the
fact that $\mathcal{A}\mathcal{D}$ has full rank, Theorem 8.2.7 of
\cite{KS88} provides the estimates
\begin{equation}
\lvert\Delta d\rvert\leq\frac{\lVert(\mathcal{A}P)^{+}\rVert}{1-\omega_{\Delta}}\left\{ \lVert\mathfrak{R}\rVert\left[\vert d\rvert+\lVert(\mathcal{A}P)^{+}\rVert\lvert\mathfrak{r}\rvert\right]+\lvert\Delta r\rvert\right\} \label{eq:e1}
\end{equation}
and
\begin{align}
\frac{\lvert\Delta d\rvert}{\lvert d\vert}\leq & \frac{1}{1-\omega_{\Delta}}\left\{ \left[\kappa_{\mathcal{C}}(\mathcal{A})+\frac{\lvert\mathfrak{r}\rvert}{\lVert\mathcal{A}\mathcal{D}\rVert\lvert d\rvert}\kappa_{\mathcal{C}}(\mathcal{A})^{2}\right]\frac{\lVert\mathfrak{R}\rVert}{\lVert\mathcal{A}\mathcal{D}\rVert}\right.\nonumber \\
 & \left.+\frac{\lVert(\mathcal{A}\mathcal{D})^{+}\rVert\lvert r\rvert}{\rvert d\rvert}\cdot\frac{\vert\Delta r\rvert}{\lvert r\rvert}\right\} .\label{eq:e2}
\end{align}
with $\mathfrak{r}=r-\mathcal{A}c$.
\begin{thm}
\label{thm:errest}Let $\lVert\Delta\mathcal{A}\rVert$ and $\lVert\Delta\mathcal{C}\rVert$
be sufficiently small such that (\ref{eq:omegaD}) and $\varkappa=\lVert\mathcal{C}^{+}\lVert\rVert\Delta\mathcal{C}\rVert<1/2$
hold true. Then it holds
\[
\lvert\Delta c\rvert\leq\frac{\lVert(\mathcal{A}P)^{+}\rVert}{1-\omega_{\Delta}}\left\{ \lVert\mathfrak{R}\rVert\left[\lvert c\rvert+\lVert(\mathcal{A}P)^{+}\rVert\lvert\mathfrak{r}\lvert\right]+\lvert\Delta r\rvert\right\} +\frac{\lVert\mathcal{C}^{+}\rVert}{1-\varkappa}\left[\sqrt{2}\kappa(\mathcal{C})+1\right]\lVert\Delta\mathcal{C}\rVert\lvert c\rvert
\]
and
\begin{align*}
\frac{\lvert\Delta c\rvert}{\lvert c\rvert}\leq & \frac{1}{1-\omega_{\Delta}}\left\{ \left[\kappa_{\mathcal{C}}(\mathcal{A})+\frac{\lvert\mathfrak{r}\rvert}{\lVert\mathcal{A}P\rVert\lvert c\rvert}\kappa_{\mathcal{C}}(\mathcal{A})^{2}\right]\frac{\lVert\mathfrak{R}\rVert}{\lVert\mathcal{A}P\rVert}\right.\\
 & \left.+\frac{\lVert(\mathcal{A}P)^{+}\rVert\lvert r\rvert}{\lvert c\rvert}\cdot\frac{\vert\Delta r\rvert}{\lvert r\rvert}\right\} +\frac{\lVert\mathcal{C}^{+}\rVert}{1-\varkappa}\left[\sqrt{2}\kappa(\mathcal{C})+1\right]\lVert\Delta\mathcal{C}\rVert.
\end{align*}
\end{thm}
\begin{proof}
It holds $c=\mathcal{D}d$ and $\Delta c=\mathcal{D}_{\Delta}\Delta d+(\mathcal{D}_{\Delta}-\mathcal{D})d$
such that $\lvert\Delta c\rvert\leq\lvert\Delta d\rvert+\lVert P_{\Delta}-P\rVert\lvert d\rvert$.
Inserting this estimate in (\ref{eq:e1}) and (\ref{eq:e2}) and using
$\lvert c\rvert=\lvert\mathcal{D}d\rvert=\lvert d\rvert$ provides
the claim.
\end{proof}
\begin{rem}
$\lvert\mathfrak{r}\lvert$ is a measure for the accuracy of the discrete
solution. Let $x_{\pi}\in X_{\pi}$ denote the discrete solution obtained
by minimizing $\Phi_{\pi,M}$ (\ref{eq:PhiM}). Its representation
becomes $c=\mathcal{R}^{-1}x_{\pi}$. Then it holds $\lvert\mathfrak{r}\rvert^{2}=\lvert\mathcal{A}c-r\rvert^{2}=\Phi_{\pi,M}(x_{\pi})$.
Hence, $\Phi_{\pi,M}(x_{\pi})\leq2(\Phi_{\pi,M}(x_{\ast})+\Phi_{\pi,M}(x_{\pi}-x_{\ast}))$.
Under the conditions of Theorem~\ref{thm:Convergence}, it holds,
therefore, $\lvert\mathfrak{r}\rvert\leq ch^{N-\mu+1}$.\qed
\end{rem}
The critical quantities to estimate the influence of perturbations
are $\kappa_{\mathcal{C}}(\mathcal{A})$ and $\lVert\mathcal{C}^{+}\rVert$,
$\kappa(\mathcal{C})$ as well as $\lVert(\mathcal{A}P)^{+}\rVert$.
The norms of $\mathcal{C}$ and its pseudoinverse depend only on the
choice of $X_{\pi}$ and the basis chosen for it, but not on the DAE.
It holds $\lVert\mathcal{C}\rVert=\sigma_{\max}(\mathcal{C})$ and
$\lVert\mathcal{C}^{+}\rVert=\sigma_{\min}(\mathcal{C})^{-1}$ with
$\sigma_{\min}(\mathcal{C})$ being the smallest nonvanishing singular
value of $\mathcal{C}$. Since $\mathcal{C}$ has full row rank, $\sigma_{\min}(\mathcal{C})=\left(\lambda_{\min}(\mathcal{C}\mathcal{C}^{T})\right)^{1/2}$
and $\sigma_{\max}(\mathcal{C})=\left(\lambda_{\max}(\mathcal{C}\mathcal{C}^{T})\right)^{1/2}$. 

With $\mathcal{C}$ from (\ref{eq:discconstraint}) we observe that
\[
\mathcal{C}=\Pi_{1}\left[I_{k}\otimes\mathcal{C}_{\textrm{s}}\vert\mathcal{O}_{\textrm{s}}\right]\Pi_{2}
\]
with
\[
\mathcal{C}_{\textrm{s}}=\begin{bmatrix}\bar{\mathcal{P}}_{1}(t_{1}) & -\bar{\mathcal{P}}{}_{2}(t_{1})\\
 & \bar{\mathcal{P}}_{2}(t_{2}) & -\bar{\mathcal{P}_{3}}(t_{2})\\
 &  & \ddots & \ddots\\
 &  &  & \ddots & \ddots\\
 &  &  &  & \bar{\mathcal{P}}_{n-1}(t_{n-1}) & -\bar{\mathcal{P}}_{n}(t_{n-1})
\end{bmatrix}\in\R^{(n-1)\times n(N+1)}
\]
and $\mathcal{O}_{\textrm{s}}\in\R^{k(n-1)\times nN(m-k)}$ consists
entirely of zero elements. The permutation matrices $\Pi_{1}$ and
$\Pi_{2}$ are are constructed as follows: Let $x=[x_{1},x_{2},\ldots,x_{m}]^{T}\in\tilde{X}_{\pi}$.
First, the equations in $\mathcal{C}c=0$ are reordered such that
first all equations related to the first component $x_{1}$, then
those of $x_{2}$, and so on until $x_{k}$ are available. This reordering
is expressed via $\Pi_{1}$. The column permutation $\Pi_{2}$ reorders
the coefficients such that the ones describing the differential components
are taken first, and then the ones belonging to the algebraic components.
In particular, the coefficients $c^{\kappa}$ describing $x_{\kappa}$
are given by $c^{\kappa}=[c_{1\kappa0},c_{1,\kappa1},\ldots,c_{1\kappa N},c_{2\kappa0},\ldots,c_{n\kappa N}]^{T}$.
Then we have
\begin{align}
\mathcal{C}\mathcal{C}^{T} & =\Pi_{1}\left[I_{k}\otimes\mathcal{C}_{\textrm{s}}\vert\mathcal{O}_{\textrm{s}}\right]\Pi_{2}\Pi_{2}^{T}\left[\begin{array}{c}
I_{k}\otimes\mathcal{C}_{\textrm{s}}^{T}\\
\mathcal{O}_{\textrm{s}}^{T}
\end{array}\right]\Pi_{1}^{T}=\Pi_{1}(I_{k}\otimes\mathcal{C_{\textrm{s}}}\mathcal{C}_{\textrm{s}}^{T})\Pi_{1}^{T},\label{eq:CCT-1}
\end{align}
Using (\ref{eq:barPj}) and (\ref{eq:barPj2}), it holds
\[
\mathcal{C}_{\textrm{s}}=C_{\textrm{s}}\diag(h_{1},\ldots,h_{n})
\]
with
\begin{equation}
C_{\textrm{s}}=\left[\begin{array}{cccccc}
f & -e_{1}^{T}\\
 & f & -e_{1}^{T}\\
 &  & \ddots & \ddots\\
\\
 &  &  &  & f & -e_{1}^{T}
\end{array}\right]\label{eq:Cs-1}
\end{equation}
where $e_{1}$ is the first unit vector and $f=[1,\int_{0}^{1}p_{0}(\sigma)\dt\sigma,\ldots,\int_{0}^{1}p_{N-1}(\sigma)\dt\sigma]$.
This leads to
\begin{equation}
\mathcal{C}_{\textrm{s}}\mathcal{C}_{\textrm{s}}^{T}=\left[\begin{array}{ccccc}
h_{1}^{2}\lvert f\rvert^{2}+h_{2}^{2} & -h_{2}^{2}\\
-h_{2}^{2} & h_{2}^{2}\lvert f\rvert^{2}+h_{3}^{2} & -h_{3}^{2}\\
 & \ddots & \ddots & \ddots\\
\\
 &  &  & -h_{n-1}^{2} & h_{n-1}^{2}\lvert f\rvert^{2}+h_{n}^{2}
\end{array}\right].\label{eq:CsCsT-1}
\end{equation}
The eigenvalues of $\mathcal{C}\mathcal{C}^{T}$ are those of (\ref{eq:CsCsT-1}).
For constant stepsize $h$, this reduces to $\mathcal{C}_{\textrm{s}}\mathcal{C}_{\textrm{s}}^{T}=h^{2}C_{s}C_{s}^{T}$,
which is a Toeplitz tridiagonal matrix. In this case, the eigenvalues
of $C_{s}C_{s}^{T}$ are given by \cite[Theorem 2.2]{KuScTs99} 
\begin{equation}
\lambda_{j}=1+\lvert f\rvert^{2}-2\cos\left(\frac{j\pi}{n}\right),\quad j=1,\ldots,n-1.\label{eq:EV}
\end{equation}

\begin{prop}
\label{propCcond}Let the grid (\ref{eq:mesh}) be equidistant with
stepsize $h$, and $C_{s}$ be given by (\ref{eq:Cs-1}). Then it
holds
\begin{itemize}
\item For the Legendre basis $1\leq\lambda_{\min}(C_{s}C_{s}^{T})\leq\lambda_{\max}(C_{s}C_{s}^{T})\leq5$;
\item For the modified Legendre basis $2N\leq\lambda_{\min}(C_{s}C_{s}^{T})\leq\lambda_{\max}(C_{s}C_{s}^{T})\leq 2N+6$;
\item For the Chebyshev basis $1\leq\lambda_{\min}(C_{s}C_{s}^{T})\leq\lambda_{\max}(C_{s}C_{s}^{T})\leq4+2\ln2$.
\item For the Runge-Kutta basis assume additionally that $\int_{0}^{1}p_{i}(\sigma)\dt\sigma\geq0$,
$i=0,1,\ldots,N-1$. Then $N^{-1}\leq\lambda_{\min}(C_{s}C_{s}^{T})\leq\lambda_{\max}(C_{s}C_{s}^{T})\leq5$.
\end{itemize}
\end{prop}
\begin{proof}
In the case of the Legendre basis, it holds $f=[1,1,0,\ldots,0]$.
Hence, $\lvert f\rvert^{2}=2$ such that the statement follows.

For the modified Legendre basis, we have $f=[1,2,0,2,0,\ldots]$ such
that
\[
\lvert f\rvert^{2}=\begin{cases}
2N+1, & N\text{ even},\\
2N+3, & N\text{ odd}.
\end{cases}
\]

For the Chebyshev basis, we observe 
\[
\int_{0}^{1}p_{i}(\sigma)\dt\sigma=\begin{cases}
\frac{1}{2}\frac{1+(-1)^{i}}{1-i^{2}}, & i\neq1,\\
0, & i=1.
\end{cases}
\]
This leads to $f=[1,1,0,-\frac{1}{3},0,-\frac{1}{8},0,\ldots]$. Hence,
\[
2\leq\lvert f\rvert^{2}\leq2+\sum_{i=1}^{\infty}\left(\frac{1}{1-(2i)^{2}}\right)^{2}\leq2+\sum_{i=1}^{\infty}\frac{1}{i(4i^{2}-1)}=2+2\ln2-1.
\]
For the sum of the series, cf. \cite[p 269, series 110.d]{Knopp54}.
This provides the estimate for the Chebyshev basis.

In case of the Runge-Kutta basis it holds $\sum_{i=0}^{N-1}p_{i}(\sigma)\equiv1$.
With $f=[1,f_{2},\ldots,f_{N+1}]$ it holds then $f_{i}\geq0$ and
$\sum_{i=2}^{N+1}f_{i}=1$. Hence,
\[
\frac{1}{N}=\frac{1}{N}\left(\sum_{i=2}^{N+1}f_{i}\right)^{2}\leq\sum_{i=2}^{N+1}f_{i}^{2}\leq\sum_{i=2}^{N+1}f_{i}=1.
\]
This yields $1+N^{-1}\leq\lvert f\rvert^{2}\leq2$ and the claim follows.
\end{proof}
\begin{rem}
For the Runge-Kutta basis, the values $f_{i}=\int_{0}^{1}p_{i-1}(\sigma)\dt\sigma$
are just the weights of the interpolatory quadrature rule corresponding
to the nodes $\tau_{1},\ldots,\tau_{N}$ of (\ref{eq:ChebNodes}).
For a number of common choices of nodes these weights are known to
be positive. Examples are the Gauss-Legendre nodes, Radau nodes, and
Lobatto nodes \cite[Section 2.7]{DavRab84}. It holds also true for
Chebyshev nodes and many others, see e.g. \cite[pp 85f]{DavRab84}.\qed
\end{rem}
\begin{cor}
\label{Cor1}For equidistant grids (\ref{eq:mesh}), it holds
\begin{itemize}
\item For the Legendre basis $\kappa(\mathcal{C})\leq\sqrt{5}$ and $\lVert\mathcal{C}^{+}\rVert\leq h^{-1}$;
\item For the modified Legendre basis $\kappa(\mathcal{C})\leq\left(\frac{2N+6}{2N}\right)^{1/2}$
and $\lVert\mathcal{C}^{+}\rVert\leq (2N)^{-1/2}h^{-1}$;
\item For the Chebyshev basis $\kappa(\mathcal{C})\leq(4+2\ln2)^{1/2}\approx2.32$
and $\lVert\mathcal{C}^{+}\rVert\leq h^{-1}$.
\item For the Runge-Kutta basis $\kappa(\mathcal{C})\leq(5N)^{1/2}$ and
$\lVert\mathcal{C}^{+}\rVert\leq N^{1/2}h^{-1}$ provided that $\int_{0}^{1}p_{i}(\sigma)\dt\sigma\geq0$,
$i=0,1,\ldots,N-1$. 
\end{itemize}
\end{cor}
It should be emphasized again that, if $\ker(\mathcal{C}+\Delta\mathcal{C})\neq\ker\mathcal{C}$,
it cannot be guaranteed that the solution of the perturbed problem
$\mathcal{R}(c+\Delta c)$ belongs to $X_{\pi}$. Instead, it belongs
to $\tilde{X}_{\pi}$, only. Simple projection algorithms of elements
of $\tilde{X}_{\pi}$ onto $X_{\pi}$ can be derived, see Appendix~\ref{sec:Projections}.
In our experiments so far, these projections did not have a better
accuracy than the unprojected numerical solutions.

\section{Some examples}

\subsection{Conditioning of the representation map $\mathcal{R}$\label{subsec:CondRepr}}

For each selection $\{p_{0},\ldots,p_{N-1}\}$ of basis polynomials,
the conditioning of the representation map depends both on the grid
and on $N$. For simplicity, we assume here that an equidistant grid
with stepsize $h$ is used for defining $X_{\pi}$. Besides the bases
introduced before, we will additionally consider the Runge-Kutta basis
with uniform interpolation points as used in our very first paper
on the subject \cite{HMTWW}.

The norms of the representation map and its inverse have been computed
for both settings (mapping into $L^{2}((a,b),\R^{m})$ and $H_{D}^{1}(a,b)$)
and for polynomial degrees $N=3,5,10,20$ and $h=n^{-1}$ where $N=10,20,40,80,160,320$.
These are the first observations:
\begin{itemize}
\item $\sigma_{\min}(\hat{\mathcal{U}})$ is independent of the chosen basis
and independent of $N$ for $h\leq0.1$. However, this is not true
for larger stepsizes, cf Table~\ref{tab:h-large}.
\item For every basis, $\sigma_{\max}(\mathcal{U})\approx\sigma_{\max}(\hat{\mathcal{U}})$
up to a relative error below $10^{-3}$. This coincides with the findings
of Theorem~\ref{thm:asymptotic}.
\end{itemize}
In Tables~\ref{tab:smin}--\ref{tab:kappaU}, we present more detailed
results. From these tables, we can draw the following conclusions:
\begin{itemize}
\item The asymptotic behavior with respect to the stepsize $h$ as indicated
in Theorem~\ref{thm:asymptotic} is clearly visible.
\item For both the Legendre and the Chebyshev bases, $\sigma_{\max}(\mathcal{U})$
and $\sigma_{\max}(\hat{\mathcal{U}})$ do not depend on $N$. This
is reasonable for the Legendre basis if Proposition~\ref{RemSig}
is taken into account.
\item The asymptotics of $\sigma_{\min}(\mathcal{U})$ coincides with the
results of Theorem~\ref{thm:asymptotic} and Proposition~\ref{RemSig}
for the modified Legendre basis.
\item The norm of the representation map behaves similarly for all considered
bases. Not unexpectedly, an exception is the Runge-Kutta basis for
uniform nodes, which has a much larger norm than that for other bases.
When comparing $\sigma_{\min}(\mathcal{U})$ and $\sigma_{\max}(\mathcal{U})$
for different bases, we observe that the difference between the Legendre
basis and the Chebyshev basis on one hand and the modified Legendre
basis on the other hand it seems that they have different scaling
only, but their conditioning (being the product of the norms of the
representation map and its inverse) are similar. A similar property
holds for $\hat{\mathcal{U}}$.
\item The Runge-Kutta basis has surprisingly good properties. However, this
property depends on the representation with respect to an orthogonal
polynomial basis (in the present example, Chebyshev polynomials).
Thus, it is much more expensive to work with than using Legendre or Chebyshev
bases directly.
\end{itemize}
\begin{table}

\caption{\label{tab:smin}$\sigma_{\min}(\hat{\mathcal{U}})$}

\begin{centering}
\begin{tabular}{|c|c|}
\hline 
$n=1/h$ & $\sigma_{\min}(\hat{\mathcal{U}})$\tabularnewline
\hline 
\hline 
10 & 3.16e-2\tabularnewline
\hline 
20 & 1.12e-2\tabularnewline
\hline 
40 & 3.95e-3\tabularnewline
\hline 
80 & 1.40e-3\tabularnewline
\hline 
160 & 4.94e-4\tabularnewline
\hline 
320 & 1.75e-4\tabularnewline
\hline 
\end{tabular}
\par\end{centering}
\end{table}
\begin{table}
\caption{$\sigma_{\min}(\mathcal{\hat{U}})$. \label{tab:h-large}The column
headings denote the Legendre basis (L), the modified Legendre basis
(mL), the Chebyshev basis (Ch), the Runge-Kutta basis (RK), and the
Runge-Kutta basis with uniform nodes (RKu)}

\centering{}%
\begin{tabular}{|c|c|c|c|c|c|}
\hline 
$n=1/h$ & L & mL & Ch & RK & RKu\tabularnewline
\hline 
\hline 
\multicolumn{6}{|c|}{$N=3$}\tabularnewline
\hline 
1 & 4.47e-1 & 8.56e-1 & 5.52e-1 & 4.05e-1 & 5.57e-1\tabularnewline
\hline 
3 & 1.88e-1 & 1.89e-1 & 1.88e-1 & 1.87e-1 & 1.88e-1\tabularnewline
\hline 
5 & 8.88e-2 & 8.88e-2 & 8.88e-2 & 8.88e-2 & 8.88e-2\tabularnewline
\hline 
\multicolumn{6}{|c|}{$N=5$}\tabularnewline
\hline 
1 & 3.33e-1 & 8.56e-1 & 4.31e-1 & 2.54e-1 & 4.12e-1\tabularnewline
\hline 
3 & 1.88e-1 & 1.89e-1 & 1.88e-1 & 1.46e-1 & 1.85e-1\tabularnewline
\hline 
5 & 8.88e-2 & 8.88e-2 & 8.88e-2 & 8.86e-2 & 8.87e-2\tabularnewline
\hline 
\multicolumn{6}{|c|}{$N=10$}\tabularnewline
\hline 
1 & 2.29e-1 & 8.56e-1 & 2.93e-1 & 1.31e-1 & 2.33e-1\tabularnewline
\hline 
3 & 1.32e-1 & 1.89e-1 & 1.69e-1 & 7.57e-2 & 1.35e-1\tabularnewline
\hline 
5 & 8.88e-2 & 8.88e-2 & 8.88e-2 & 5.86e-2 & 8.84e-2\tabularnewline
\hline 
\multicolumn{6}{|c|}{$N=20$}\tabularnewline
\hline 
1 & 1.60e-1 & 8.56e-1 & 2.10e-1 & 6.65e-2 & 1.50e-1\tabularnewline
\hline 
3 & 9.24e-2 & 1.89e-1 & 1.21e-1 & 3.84e-2 & 8.68e-2\tabularnewline
\hline 
5 & 7.16e-2 & 8.88e-2 & 1.21e-1 & 3.84e-2 & 8.68e-2\tabularnewline
\hline 
\end{tabular}
\end{table}

\begin{table}
\caption{$\sigma_{\min}(\mathcal{U})$. The column headings denote the Legendre
basis (L), the modified Legendre basis (mL), the Chebyshev basis (Ch),
the Runge-Kutta basis (RK), and the Runge-Kutta basis with uniform
nodes (RKu)}

\centering{}%
\begin{tabular}{|c|c|c|c|c|c|}
\hline 
$n=1/h$ & L & mL & Ch & RK & RKu\tabularnewline
\hline 
\hline 
\multicolumn{6}{|c|}{$N=3$}\tabularnewline
\hline 
10 & 1.17e-3 & 6.83e-3 & 1.35e-3 & 9.81e-4 & 1.93e-3\tabularnewline
\hline 
20 & 4.12e-4 & 2.42e-3 & 4.78e-4 & 3.47e-4 & 6.84e-4\tabularnewline
\hline 
40 & 1.46e-4 & 8.54e-4 & 1.69e-4 & 1.23e-4 & 2.42e-4\tabularnewline
\hline 
80 & 5.15e-5 & 3.02e-4 & 5.98e-5 & 4.33e-5 & 8.55e-5\tabularnewline
\hline 
160 & 1.82e-5 & 1.07e-4 & 2.11e-5 & 1.53e-5 & 3.02e-5\tabularnewline
\hline 
320 & 6.44e-6 & 3.78e-5 & 7.47e-6 & 5.42e-6 & 1.07e-5\tabularnewline
\hline 
\multicolumn{6}{|c|}{$N=5$}\tabularnewline
\hline 
10 & 4.51e-4 & 4.76e-3 & 5.06e-4 & 2.96e-4 & 1.00e-3\tabularnewline
\hline 
20 & 1.59e-4 & 1.68e-3 & 1.79e-4 & 1.04e-4 & 3.54e-4\tabularnewline
\hline 
40 & 5.63e-5 & 5.95e-4 & 6.32e-5 & 3.69e-5 & 1.25e-4\tabularnewline
\hline 
80 & 1.99e-5 & 2.10e-4 & 2.23e-5 & 1.31e-5 & 4.43e-5\tabularnewline
\hline 
160 & 7.04e-6 & 7.44e-5 & 7.90e-6 & 4.62e-6 & 1.56e-5\tabularnewline
\hline 
320 & 2.49e-6 & 2.63e-5 & 2.79e-6 & 1.63e-6 & 5.53e-6\tabularnewline
\hline 
\multicolumn{6}{|c|}{$N=10$}\tabularnewline
\hline 
10 & 1.08e-4 & 2.73e.3 & 1.10e-4 & 4.94e-5 & 2.35e-4\tabularnewline
\hline 
20 & 3.83e-5 & 9.59e-4 & 3.90e-5 & 1.75e-5 & 8.29e-5\tabularnewline
\hline 
40 & 1.36e-5 & 3.39e-4 & 1.38e-5 & 6.17e-6 & 2.93e-5\tabularnewline
\hline 
80 & 4.79e-6 & 1.20e-4 & 4.88e-6 & 2.18e-6 & 1.04e-5\tabularnewline
\hline 
160 & 1.69e-6 & 4.24e-5 & 1.72e-6 & 7.71e-7 & 3.66e-6\tabularnewline
\hline 
320 & 5.99e-7 & 1.50e-5 & 6.10e-7 & 2.73e-7 & 1.30e-6\tabularnewline
\hline 
\multicolumn{6}{|c|}{$N=20$}\tabularnewline
\hline 
10 & 2.28e-5 & 1.46e-3 & 2.30e-5 & 7.26e-6 & 5.73e-5\tabularnewline
\hline 
20 & 8.06e-6 & 5.16e-4 & 8.12e-6 & 2.57e-6 & 2.03e-5\tabularnewline
\hline 
40 & 2.85e-6 & 1.82e-4 & 2.87e-6 & 9.08e-7 & 7.17e-6\tabularnewline
\hline 
80 & 1.01e-6 & 5.45e-5 & 1.01e-6 & 3.21e-7 & 2.53e-6\tabularnewline
\hline 
160 & 3.56e-7 & 2.28e-5 & 3.59e-7 & 1.13e-7 & 8.96e-7\tabularnewline
\hline 
320 & 1.26e-7 & 8.06e-6 & 1.27e-7 & 4.01e-8 & 3.17e-7\tabularnewline
\hline 
\end{tabular}
\end{table}
\begin{table}
\caption{$\sigma_{\max}(\hat{\mathcal{U}})=\sigma_{\max}(\mathcal{U})$. The
column headings denote the Legendre basis (L), the modified Legendre
basis (mL), the Chebyshev basis (Ch), the Runge-Kutta basis (RK),
and the Runge-Kutta basis with uniform nodes (RKu)}

\centering{}%
\begin{tabular}{|c|c|c|c|c|c|}
\hline 
$n=1/h$ & L & mL & Ch & RK & RKu\tabularnewline
\hline 
\hline 
\multicolumn{6}{|c|}{$N=3$}\tabularnewline
\hline 
10 & 3.16e-1 & 1.57e+0 & 3.41e-1 & 2.19e-1 & 2.64e-1\tabularnewline
\hline 
20 & 2.24e-1 & 1.11e+0 & 2.41e-1 & 1.55e-1 & 1.86e-1\tabularnewline
\hline 
40 & 1.58e-1 & 7.87e-1 & 1.70e-1 & 1.10e-1 & 1.32e-1\tabularnewline
\hline 
80 & 1.12e-1 & 5.56e-1 & 1.20e-1 & 7.74e-2 & 9.32e-2\tabularnewline
\hline 
160 & 7.91e-2 & 3.93e-1 & 8.52e-2 & 5.48e-2 & 6.59e-2\tabularnewline
\hline 
320 & 5.59e-2 & 2.78e-1 & 6.02e-2 & 3.87e-2 & 4.66e-2\tabularnewline
\hline 
\multicolumn{6}{|c|}{$N=5$}\tabularnewline
\hline 
10 & 3.16e-1 & 2.69e+0 & 3.41e-1 & 1.74e-1 & 3.96e-1\tabularnewline
\hline 
20 & 2.24e-1 & 1.90e+1 & 2.41e-1 & 1.23e-1 & 2.80e-1\tabularnewline
\hline 
40 & 1.58e-1 & 1.35e+1 & 1.70e-1 & 8.70e-2 & 1.98e-1\tabularnewline
\hline 
80 & 1.12e-1 & 9.51e-1 & 1.20e-1 & 6.15e-2 & 1.40e-1\tabularnewline
\hline 
160 & 7.91e-2 & 6.73e-1 & 8.52e-2 & 4.35e-2 & 9.90e-2\tabularnewline
\hline 
320 & 5.59e-2 & 4.76e-1 & 6.02e-2 & 3.08e-2 & 7.00e-2\tabularnewline
\hline 
\multicolumn{6}{|c|}{$N=10$}\tabularnewline
\hline 
10 & 3.16e-1 & 6.26e+0 & 3.41e-1 & 1.25e-1 & 3.86e+0\tabularnewline
\hline 
20 & 2.24e-1 & 4.43e+0 & 2.41e-1 & 8.81e-2 & 2.73e+0\tabularnewline
\hline 
40 & 1.58e-1 & 3.13e+0 & 1.70e-1 & 6.23e-2 & 1.93e+0\tabularnewline
\hline 
80 & 1.12e-1 & 2.21e+0 & 1.20e-1 & 4.40e-2 & 1.36e+0\tabularnewline
\hline 
160 & 7.91e-2 & 1.57e+0 & 8.52e-2 & 3.11e-2 & 9.64e-1\tabularnewline
\hline 
320 & 5.59e-2 & 1.11e+0 & 6.02e-2 & 2.20e-2 & 6.82e-1\tabularnewline
\hline 
\multicolumn{6}{|c|}{$N=20$}\tabularnewline
\hline 
10 & 3.16e-1 & 1.60e+1 & 3.41e-1 & 8.85e-2 & 1.47e+3\tabularnewline
\hline 
20 & 2.24e-1 & 1.13e+1 & 2.41e-1 & 6.26e-2 & 1.04e+3\tabularnewline
\hline 
40 & 1.58e-1 & 7.98e+0 & 1.70e-1 & 4.43e-2 & 7.37e+2\tabularnewline
\hline 
80 & 1.12e-1 & 5.64e+0 & 1.20e-1 & 3.13e-2 & 5.21e+2\tabularnewline
\hline 
160 & 7.91e-2 & 3.99e+0 & 8.52e-2 & 2.21e-2 & 3.68e+2\tabularnewline
\hline 
320 & 5.59e-2 & 2.82e+0 & 6.02e-2 & 1.56e-2 & 2.61e+2\tabularnewline
\hline 
\end{tabular}
\end{table}
\begin{table}
\caption{$\kappa(\hat{\mathcal{U}})=\sigma_{\max}(\hat{\mathcal{U}})/\sigma_{\min}(\hat{\mathcal{U}})$.
The column headings denote the Legendre basis (L), the modified Legendre
basis (mL), the Chebyshev basis (Ch), the Runge-Kutta basis (RK),
and the Runge-Kutta basis with uniform nodes (RKu)}

\centering{}%
\begin{tabular}{|c|c|c|c|c|c|}
\hline 
$n=1/h$ & L & mL & Ch & RK & RKu\tabularnewline
\hline 
\hline 
\multicolumn{6}{|c|}{$N=3$}\tabularnewline
\hline 
10 & 1.03e+1 & 4.98e+1 & 1.08e+1 & 6.95e+0 & 8.35e+0\tabularnewline
\hline 
20 & 2.00e+1 & 9.95e+1 & 2.16e+1 & 1.39e+1 & 1.67e+1\tabularnewline
\hline 
40 & 4.00e+1 & 1.99e+2 & 4.31e+1 & 2.77e+1 & 3.34e+1\tabularnewline
\hline 
80 & 8.00e+1 & 3.98e+2 & 8.62e+1 & 5.54e+1 & 6.67e+1\tabularnewline
\hline 
160 & 1.60e+2 & 7.96e+2 & 1.72e+2 & 1.11e+2 & 1.33e+2\tabularnewline
\hline 
320 & 3.20e+2 & 1.59e+3 & 3.45e+2 & 2.22e+2 & 2.67e+2\tabularnewline
\hline 
\multicolumn{6}{|c|}{$N=5$}\tabularnewline
\hline 
10 & 1.00e+1 & 8.53e+1 & 1.08e+1 & 5.52e+0 & 1.25e+1\tabularnewline
\hline 
20 & 2.00e+1 & 1.70e+2 & 2.16e+1 & 1.10e+1 & 2.51e+1\tabularnewline
\hline 
40 & 4.00e+1 & 3.40e+2 & 4.31e+1 & 2.20e+1 & 5.01e+1\tabularnewline
\hline 
80 & 8.00e+1 & 6.81e+2 & 8.62e+1 & 4.40e+1 & 1.00e+2\tabularnewline
\hline 
160 & 1.60e+2 & 1.36e+3 & 1.72e+2 & 8.80e+1 & 2.00e+2\tabularnewline
\hline 
320 & 3.20e+2 & 2.72e+3 & 3.45e+2 & 1.76e+2 & 4.01e+2\tabularnewline
\hline 
\multicolumn{6}{|c|}{$N=10$}\tabularnewline
\hline 
10 & 1.00e+1 & 1.98e+2 & 1.08e+1 & 3.95e+0 & 1.22e+2\tabularnewline
\hline 
20 & 2.00e+1 & 3.96e+2 & 2.16e+1 & 7.88e+0 & 2.44e+2\tabularnewline
\hline 
40 & 4.00e+1 & 7.92e+2 & 4.31e+1 & 1.58e+1 & 4.88e+2\tabularnewline
\hline 
80 & 8.00e+1 & 1.58e+3 & 8.62e+1 & 3.15e+1 & 9.76e+2\tabularnewline
\hline 
160 & 1.60e+2 & 3.17e+3 & 1.72e+2 & 6.30e+1 & 1.95e+3\tabularnewline
\hline 
320 & 3.20e+2 & 6.34e+3 & 3.45e+2 & 1.26e+2 & 3.90e+3\tabularnewline
\hline 
\multicolumn{6}{|c|}{$N=20$}\tabularnewline
\hline 
10 & 1.00e+1 & 5.05e+2 & 1.08e+1 & 4.21e+0 & 4.67e+4\tabularnewline
\hline 
20 & 2.00e+1 & 1.01e+3 & 2.16e+1 & 5.60e+0 & 9.32e+4\tabularnewline
\hline 
40 & 4.00e+1 & 2.02e+3 & 4.31e+1 & 1.12e+1 & 1.86e+5\tabularnewline
\hline 
80 & 8.00e+1 & 4.04e+3 & 8.62e+1 & 2.24e+1 & 3.73e+5\tabularnewline
\hline 
160 & 1.60e+2 & 8.07e+3 & 1.72e+2 & 4.48e+1 & 7.46e+5\tabularnewline
\hline 
320 & 3.20e+2 & 1.61e+4 & 3.45e+2 & 9.00e+1 & 1.49e+6\tabularnewline
\hline 
\end{tabular}
\end{table}
\begin{table}
\caption{$\kappa(\mathcal{U})=\sigma_{\max}(\mathcal{U})/\sigma_{\min}(\mathcal{U})$.
The column headings denote the Legendre basis (L), the modified Legendre
basis (mL), the Chebyshev basis (Ch), the Runge-Kutta basis (RK),
and the Runge-Kutta basis with uniform nodes (RKu)\label{tab:kappaU}}

\centering{}%
\begin{tabular}{|c|c|c|c|c|c|}
\hline 
$n=1/h$ & L & mL & Ch & RK & RKu\tabularnewline
\hline 
\hline 
\multicolumn{6}{|c|}{$N=3$}\tabularnewline
\hline 
10 & 2.71e+2 & 2.30e+2 & 2.52e+2 & 2.23e+2 & 1.36e+2\tabularnewline
\hline 
20 & 5.43e+2 & 4.61e+2 & 5.04e+2 & 4.47e+2 & 2.73e+2\tabularnewline
\hline 
40 & 1.09e+3 & 9.21e+2 & 1.01e+3 & 8.93e+2 & 5.45e+2\tabularnewline
\hline 
80 & 2.17e+3 & 1.84e+3 & 2.01e+3 & 1.79e+3 & 1.09e+3\tabularnewline
\hline 
160 & 4.34e+3 & 3.68e+3 & 4.03e+3 & 3.57e+3 & 2.18e+3\tabularnewline
\hline 
320 & 8.68e+3 & 7.37e+3 & 8.06e+3 & 7.15e+3 & 4.36e+3\tabularnewline
\hline 
\multicolumn{6}{|c|}{$N=5$}\tabularnewline
\hline 
10 & 7.03e+2 & 5.65e+2 & 6.74e+2 & 5.89e+2 & 3.96e+2\tabularnewline
\hline 
20 & 1.40e+3 & 1.13e+3 & 1.35e+3 & 1.18e+3 & 7.91e+2\tabularnewline
\hline 
40 & 2.81e+3 & 2.26e+3 & 2.70e+3 & 2.36e+3 & 1.58e+3\tabularnewline
\hline 
80 & 5.61e+3 & 4.52e+3 & 5.39e+3 & 4.71e+3 & 3.16e+3\tabularnewline
\hline 
160 & 1.12e+4 & 9.05e+3 & 1.08e+4 & 9.42e+3 & 6.33e+3\tabularnewline
\hline 
320 & 2.25e+4 & 1.81e+4 & 2.16e+4 & 1.88e+4 & 1.27e+4\tabularnewline
\hline 
\multicolumn{6}{|c|}{$N=10$}\tabularnewline
\hline 
10 & 2.92e+3 & 2.31e+3 & 3.09e+3 & 2.53e+3 & 1.65e+4\tabularnewline
\hline 
20 & 5.83e+3 & 4.62e+3 & 6.18e+3 & 5.05e+3 & 3.29e+4\tabularnewline
\hline 
40 & 1.17e+4 & 9.23e+3 & 1.24e+4 & 1.01e+4 & 6.58e+4\tabularnewline
\hline 
80 & 2.33e+4 & 1.85e+4 & 2.47e+4 & 2.02e+4 & 1.32e+5\tabularnewline
\hline 
160 & 4.67e+4 & 3.69e+4 & 4.94e+4 & 4.04e+4 & 2.63e+5\tabularnewline
\hline 
320 & 9.33e+4 & 7.39e+4 & 9.88e+4 & 8.07e+4 & 5.26e+5\tabularnewline
\hline 
\multicolumn{6}{|c|}{$N=20$}\tabularnewline
\hline 
10 & 1.39e+4 & 1.09e+4 & 1.48e+4 & 1.22e+4 & 2.57e+7\tabularnewline
\hline 
20 & 2.78e+4 & 2.19e+4 & 2.97e+4 & 2.44e+4 & 5.14e+7\tabularnewline
\hline 
40 & 5.55e+4 & 4.37e+4 & 5.94e+4 & 4.87e+4 & 1.03e+8\tabularnewline
\hline 
80 & 1.11e+5 & 8.74e+4 & 1.19e+5 & 9.75e+4 & 2.06e+8\tabularnewline
\hline 
160 & 2.22e+5 & 1.75e+5 & 2.37e+5 & 1.95e+5 & 4.11e+8\tabularnewline
\hline 
320 & 4.44e+5 & 3.50e+5 & 4.75e+5 & 3.90e+5 & 8.22e+8\tabularnewline
\hline 
\end{tabular}
\end{table}

\subsection{Conditioning of the constrained minimization problems}

In order to provide a first insight into the conditioning of the constrained
minimization problem (\ref{eq:discmin})-(\ref{eq:discconstraint}),
we computed the condition numbers $\kappa_{\mathcal{C}}(\mathcal{A})$
which have a crucial importance for the behavior of the computational
error. Discussions of $\kappa(\mathcal{C})$ and $\lVert\mathcal{C}^{+}\rVert$
have been provided earlier (Proposition~\ref{propCcond} and Corollary~\ref{Cor1}).
The examples below are chosen from our earlier investigations that led to
surprisingly accurate results.

As done before, we use the bases as introduced in Section~\ref{subsec:CondRepr}.
We abandon the use the Runge-Kutta basis with uniform nodes since
this basis has a bad conditioning. We choose $M=N+1$ and the Gauss-Legendre
nodes as collocation points (\ref{eq:nodes}). For this choice, $\Phi_{\pi,M}^{R}=\Phi_{\pi,M}^{I}$
(see (\ref{LR}), (\ref{LI})) and $\kappa_{\mathcal{C}}(\mathcal{A})$
is identical for both choices.
\begin{exam}
The first example is an index-3 DAE without dynamic degrees of freedom.
It has been used before in numerous papers, e.g., \cite{HMTWW,HMT,HM_Part1}.
The problem is given by 
\begin{align*}
x'_{2}(t)+x_{1}(t) & =q_{1}(t),\\
t\eta x'_{2}(t)+x'_{3}(t)+(\eta+1)x_{2}(t) & =q_{2}(t),\\
t\eta x_{2}(t)+x_{3}(t) & =q_{3}(t),\quad t\in[0,1].
\end{align*}
For unique solvability, no boundary or initial conditions are necessary.
We choose the exact solution 
\begin{align*}
x_{\ast,1}(t) & =e^{-t}\sin t,\\
x_{\ast,2}(t) & =e^{-2t}\sin t,\\
x_{\ast,3}(t) & =e^{-t}\cos t
\end{align*}
and adapt the right-hand side $q$ accordingly.
In Table~\ref{tab:ExB-1}, the values of $\kappa_{\mathcal{C}}(\mathcal{A})$
for $\Phi_{\pi,M}^{R}$ and $\Phi_{\pi,M}^{C}$ are provided. It turns
out that the behavior for different functionals is comparable. Therefore,
in the following examples, we present only the values for $\Phi_{\pi,M}^{R}$.\qed
\end{exam}
\begin{table}
\caption{$\kappa_{\mathcal{C}}(\mathcal{A})$ for $\Phi_{\pi,M}^{R}$ and $\Phi_{\pi,M}^{C}$.
Here, L denotes the Legendre basis, mL the modified Legendre basis,
Ch the Chebyshev basis, and RK the Runge-Kutta basis\label{tab:ExB-1}.
The smallest values are set in boldface}

\centering{}%
\begin{turn}{90}
\begin{tabular}{|c|c|c|c|c|c|c|c|c|}
\hline 
$n=1/h$ & \multicolumn{4}{c|}{$\kappa_{\mathcal{C}}(\mathcal{A})$ for $\Phi_{\pi,M}^{R}$ } & \multicolumn{4}{c|}{$\kappa(\mathcal{A})$ for $\Phi_{\pi,M}^{C}$}\tabularnewline
\hline 
 & L & mL & Ch & RK & L & mL & Ch & RK\tabularnewline
\hline 
\multicolumn{9}{|c|}{$N=3$}\tabularnewline
\hline 
10 & 5.77e+4 & 5.76e+4 & 6.22e+4 & \textbf{4.96e+4} & 6.01e+4 & 7.04e+4 & 6.04e+4 & \textbf{4.53e+4}\tabularnewline
\hline 
20 & 2.37e+5 & 2.40e+5 & 2.55e+5 & \textbf{2.03e+5} & 2.47e+5 & 2.93e+5 & 2.48e+5 & \textbf{1.85e+5}\tabularnewline
\hline 
40 & 9.62e+5 & 9.79e+5 & 1.04e+6 & \textbf{8.25e+5} & 1.00e+6 & 1.20e+6 & 1.01e+6 & \textbf{7.52e+5}\tabularnewline
\hline 
80 & 3.88e+6 & 3.96e+6 & 4.18e+6 & \textbf{3.32e+6} & 4.05e+6 & 4.86e+6 & 4.06e+6 & \textbf{3.03e+6}\tabularnewline
\hline 
\multicolumn{9}{|c|}{$N=5$}\tabularnewline
\hline 
10 & 4.41e+5 & \textbf{2.79e+5} & 4.69e+5 & 3.84e+5 & 4.58e+5 & 4.06e+5 & 4.43e+5 & \textbf{3.33e+5}\tabularnewline
\hline 
20 & 1.80e+6 & \textbf{1.23e+6} & 1.91e+6 & 1.56e+6 & 1.86e+6 & 1.68e+6 & 1.80e+6 & \textbf{1.34e+6}\tabularnewline
\hline 
40 & 7.25e+6 & \textbf{5.02e+6} & 7.71e+6 & 6.32e+6 & 7.52e+6 & 6.85e+6 & 7.27e+6 & \textbf{5.39e+6}\tabularnewline
\hline 
80 & 2.92e+7 & \textbf{2.03e+7} & 3.10e+7 & 2.54e+7 & 3.02e+7 & 2.77e+7 & 2.92e+7 & \textbf{2.16e+7}\tabularnewline
\hline 
\multicolumn{9}{|c|}{$N=10$}\tabularnewline
\hline 
10 & 7.16e+6 & \textbf{3.92e+6} & 7.71e+6 & 6.50e+6 & 7.02e+6 & 6.09e+6 & 6.73e+6 & \textbf{5.11e+6}\tabularnewline
\hline 
20 & 2.89e+7 & \textbf{1.59e+7} & 3.11e+7 & 2.64e+7 & 2.84e+7 & 2.46e+7 & 2.71e+7 & \textbf{2.04e+7}\tabularnewline
\hline 
40 & 1.16e+8 & \textbf{6.39e+7} & 1.25e+8 & 1.06e+8 & 1.14e+8 & 9.92e+7 & 1.09e+8 & \textbf{8.17e+7}\tabularnewline
\hline 
80 & 4.67e+8 & \textbf{2.57e+8} & 5.02e+8 & 4.27e+8 & 4.58e+8 & 3.98e+8 & 4.37e+8 & \textbf{3.27e+8}\tabularnewline
\hline 
\multicolumn{9}{|c|}{$N=20$}\tabularnewline
\hline 
10 & 1.34e+8 & \textbf{6.79e+7} & 1.49e+8 & 1.23e+8 & 1.17e+8 & 1.21e+8 & 1.13e+8 & \textbf{8.40e+7}\tabularnewline
\hline 
20 & 5.39e+8 & \textbf{2.25e+8} & 5.99e+8 & 4.98e+8 & 4.71e+8 & 4.89e+8 & 4.56e+8 & \textbf{3.34e+8}\tabularnewline
\hline 
40 & 2.16e+9 & \textbf{1.10e+9} & 2.41e+9 & 2.01e+9 & 1.89e+9 & 1.97e+9 & 1.83e+9 & \textbf{1.34e+9}\tabularnewline
\hline 
80 & 8.67e+9 & \textbf{4.42e+9} & 9.65e+9 & 8.06e+9 & 7.59e+9 & 7.89e+9 & 7.34e+9 & \textbf{5.34e+9}\tabularnewline
\hline 
\end{tabular}
\end{turn}
\end{table}

\begin{exam}
We continue with an example of a Hessenberg index-2 system used previously
in \cite{HMTWW}. Consider the DAE system 
\begin{align*}
x'_{1}(t)+\lambda x_{1}(t)-x_{2}(t)-x_{3}(t) & =q_{1}(t),\\
x'_{2}(t)+(\eta t(1-\eta t)-\eta)x_{1}(t)+\lambda x_{2}(t)-\eta tx_{3}(t) & =q_{2}(t),\\
(1-\eta t)x_{1}(t)+x_{2}(t) & =q_{3}(t),\quad t\in[0,1],
\end{align*}
with the right hand side $q$ chosen in such a way that 
\begin{align*}
x_{1}(t) & =e^{-t}\sin t,\\
x_{2}(t) & =e^{-2t}\sin t,\\
x_{3}(t) & =e^{-t}\cos t,
\end{align*}
is a solution. It has one dynamical degree of freedom. We choose the
special condition 
\[
x_{1}(0)=0.
\]
The results for $\eta=-25$ and $\lambda=-1$ are provided in the
Table~\ref{tab:ExD}.\qed
\end{exam}
\begin{table}

\caption{$\kappa_{\mathcal{C}}(\mathcal{A})$ for $\Phi_{\pi,M}^{R}$. Here,
L denotes the Legendre basis, mL the modified Legendre basis, Ch the
Chebyshev basis, and RK the Runge-Kutta basis\label{tab:ExD}. The
smallest values are set in boldface}

\begin{centering}
\begin{tabular}{|c|c|c|c|c|}
\hline 
$n=1/h$ & \multicolumn{4}{c|}{$\kappa_{\mathcal{C}}(\mathcal{A})$}\tabularnewline
\hline 
 & L & mL & Ch & RK\tabularnewline
\hline 
\multicolumn{5}{|c|}{$N=3$}\tabularnewline
\hline 
10 & 1.95e+5 & 3.42e+5 & 2.00e+5 & \textbf{1.96e+5}\tabularnewline
\hline 
20 & 2.56e+5 & 6.90e+5 & 2.65e+5 & \textbf{2.48e+5}\tabularnewline
\hline 
40 & 4.01e+5 & 1.58e+6 & 4.22e+5 & \textbf{3.50e+5}\tabularnewline
\hline 
80 & 8.17e+5 & 3.82e+6 & 8.73e+5 & \textbf{6.06e+5}\tabularnewline
\hline 
\multicolumn{5}{|c|}{$N=5$}\tabularnewline
\hline 
10 & 6.23e+5 & \textbf{5.25e+5} & 6.13e+5 & 7.05e+5\tabularnewline
\hline 
20 & 8.54e+5 & 1.19e+6 & \textbf{8.48e+5} & 9.32e+5\tabularnewline
\hline 
40 & 1.31e+6 & 2.75e+6 & 1.32e+6 & \textbf{1.26e+6}\tabularnewline
\hline 
80 & 2.36e+6 & 6.61e+6 & 2.41e+6 & \textbf{2.03e+6}\tabularnewline
\hline 
\multicolumn{5}{|c|}{$N=10$}\tabularnewline
\hline 
10 & 3.06e+6 & \textbf{1.33e+6} & 3.02e+6 & 4.45e+6\tabularnewline
\hline 
20 & 4.28e+6 & \textbf{3.02e+6} & 4.19e+6 & 5.98e+6\tabularnewline
\hline 
40 & 6.63e+6 & 6.85e+6 & \textbf{6.55e+6} & 7.78e+6\tabularnewline
\hline 
80 & 1.19e+7 & 1.61e+7 & 1.19e+7 & \textbf{1.14e+7}\tabularnewline
\hline 
\multicolumn{5}{|c|}{$N=20$}\tabularnewline
\hline 
10 & 1.68e+7 & \textbf{4.73e+6} & 1.71e+7 & 3.05e+7\tabularnewline
\hline 
20 & 2.12e+7 & \textbf{1.03e+7} & 2.18e+7 & 3.77e+7\tabularnewline
\hline 
40 & 3.23e+7 & \textbf{2.23e+7} & 3.35e+7 & 4.76e+7\tabularnewline
\hline 
80 & 6.12e+7 & \textbf{4.93e+7} & 6.37e+7 & 6.74e+7\tabularnewline
\hline 
\end{tabular}
\par\end{centering}
\end{table}

\begin{exam}
Our next example is a linearized problem proposed by Campbell\&More
\cite{CampbellMoore95}. It has been used previously in the experiments
in \cite{HMT,HM_Part1,HM_Part2} and others. Let
\[
A(Dx)'(t)+B(t)x(t)=q(t),\quad t\in[0,5],
\]
where 
\begin{align*}
A=\begin{bmatrix}1 & 0 & 0 & 0 & 0 & 0\\
0 & 1 & 0 & 0 & 0 & 0\\
0 & 0 & 1 & 0 & 0 & 0\\
0 & 0 & 0 & 1 & 0 & 0\\
0 & 0 & 0 & 0 & 1 & 0\\
0 & 0 & 0 & 0 & 0 & 1\\
0 & 0 & 0 & 0 & 0 & 0
\end{bmatrix},D=\begin{bmatrix}1 & 0 & 0 & 0 & 0 & 0 & 0\\
0 & 1 & 0 & 0 & 0 & 0 & 0\\
0 & 0 & 1 & 0 & 0 & 0 & 0\\
0 & 0 & 0 & 1 & 0 & 0 & 0\\
0 & 0 & 0 & 0 & 1 & 0 & 0\\
0 & 0 & 0 & 0 & 0 & 1 & 0
\end{bmatrix},
\end{align*}
\begin{align*}
B(t)=\begin{bmatrix}0 & 0 & 0 & -1 & 0 & 0 & 0\\
0 & 0 & 0 & 0 & -1 & 0 & 0\\
0 & 0 & 0 & 0 & 0 & -1 & 0\\
0 & 0 & \sin t & 0 & 1 & -\cos t & -2\rho\cos^{2}t\\
0 & 0 & -\cos t & -1 & 0 & -\sin t & -2\rho\sin t\cos t\\
0 & 0 & 1 & 0 & 0 & 0 & 2\rho\sin t\\
2\rho\cos^{2}t & 2\rho\sin t\cos t & -2\rho\sin t & 0 & 0 & 0 & 0
\end{bmatrix},\quad\rho=5,
\end{align*}
subject to the initial conditions 
\[
x_{2}(0)=1,\quad x_{3}(0)=2,\quad x_{5}(0)=0,\quad x_{6}(0)=0.
\]
This problem has index 3 and dynamical degree of freedom $l_{dyn}=4$.
The right-hand side $q$ has been chosen in such a way that the exact
solution becomes 
\begin{alignat*}{2}
x_{\ast,1} & =\sin t, & x_{\ast,4} & =\cos t,\\
x_{\ast,2} & =\cos t, & x_{\ast,5} & =-\sin t,\\
x_{\ast,3} & =2\cos^{2}t, & x_{\ast,6} & =-2\sin2t,\\
x_{\ast,7} & =-\rho^{-1}\sin t.
\end{alignat*}
The results are shown in Table~\ref{tab:ExCM}. Note that, in the
present example, $h=5/n$ in contrast to all previous computations
where $h=1/n$.\qed
\end{exam}
\begin{table}

\caption{$\kappa_{\mathcal{C}}(\mathcal{A})$ for $\Phi_{\pi,M}^{R}$. Here,
L denotes the Legendre basis, mL the modified Legendre basis, Ch the
Chebyshev basis, and RK the Runge-Kutta basis\label{tab:ExCM}. The
smallest values are set in boldface}

\centering{}%
\begin{tabular}{|c|c|c|c|c|}
\hline 
$n=5/h$ & \multicolumn{4}{c|}{$\kappa_{\mathcal{C}}(\mathcal{A})$}\tabularnewline
\hline 
 & L & mL & Ch & RK\tabularnewline
\hline 
\multicolumn{5}{|c|}{$N=3$}\tabularnewline
\hline 
10 & 4.64e+2 & 1.97e+3 & 5.00e+2 & \textbf{4.12e+2}\tabularnewline
\hline 
20 & 1.43e+3 & 5.06e+3 & 1.54e+3 & \textbf{1.21e+3}\tabularnewline
\hline 
40 & 5.543e+3 & 1.38e+4 & 5.96e+3 & \textbf{4.70e+3}\tabularnewline
\hline 
80 & 2.19e+4 & 4.12e+4 & 2.36e+4 & \textbf{1.86e+4}\tabularnewline
\hline 
\multicolumn{5}{|c|}{$N=5$}\tabularnewline
\hline 
10 & 1.63e+3 & 3.74e+3 & 1.70e+3 & \textbf{1.46e+3}\tabularnewline
\hline 
20 & 6.26e+3 & 1.21e+4 & 6.55e+3 & \textbf{5.60e+3}\tabularnewline
\hline 
40 & 2.48e+4 & 4.63e+4 & 2.60e+4 & \textbf{2.22e+4}\tabularnewline
\hline 
80 & 9.88e+4 & 1.83e+5 & 1.04e+5 & \textbf{8.85e+4}\tabularnewline
\hline 
\multicolumn{5}{|c|}{$N=10$}\tabularnewline
\hline 
10 & 3.27e+4 & 4.69e+4 & 3.58e+4 & \textbf{3.20e+4}\tabularnewline
\hline 
20 & 1.29e+5 & 1.89e+5 & 1.42e+5 & \textbf{1.18e+5}\tabularnewline
\hline 
40 & 5.14e+5 & 7.58e+5 & 5.65e+5 & \textbf{4.69e+5}\tabularnewline
\hline 
80 & 2.06e+6 & 3.03e+6 & 2.26e+6 & \textbf{1.88e+6}\tabularnewline
\hline 
\multicolumn{5}{|c|}{$N=20$}\tabularnewline
\hline 
10 & \textbf{7.30e+5} & 8.64e+5 & 8.16e+5 & 9.69e+5\tabularnewline
\hline 
20 & 2.91e+6 & 3.63e+6 & 3.26e+6 & \textbf{2.71e+6}\tabularnewline
\hline 
40 & 1.17e+7 & 1.50e+7 & 1.31e+7 & \textbf{1.09e+7}\tabularnewline
\hline 
80 & 4.69e+7 & 6.10e+7 & 5.25e+7 & \textbf{4.36e+7}\tabularnewline
\hline 
\end{tabular}
\end{table}

The numerical experiments give rise to the following observations:
\begin{itemize}
\item The condition numbers of the discrete problem have almost the same
size for given polynomial degree $N$ and stepsize $h$.
\item The experiments indicate that the Runge-Kutta basis seems to provide
the lowest condition number for smaller stepsizes. In the case of
higher order ansatz functions and larger stepsizes, the modified Legendre
basis seems to provides smallest condition numbers.
\item In order to obtain a complete picture of the relative merits of the
different bases, in the case discussed in Theorem~\ref{thm:errest},
not only the condition number $\kappa(\mathcal{C})$ of $\mathcal{C}$
but the term $\lVert\mathcal{C}^{+}\rVert\kappa(\mathcal{C})$has
to be taken into account. Corollary~\ref{Cor1} shows that the modified
Legendre basis is well-suited for higher orders $N$.
\item If the perturbed solution $\tilde{c}$ of (\ref{eq:pertmini}) is
projected back onto the nullspace $\ker\mathcal{C}$, e.g., by one
of the methods developed in Appendix~\ref{sec:Projections}, we can
assume that the conditions of Theorem~\ref{Theo:constant-nullspace}
are fulfilled. In this case, $\mathcal{C}$ does not have any influence
on the error estimation.
\end{itemize}

\section{Conclusions}

In this paper, we investigated the conditioning of the discrete problems
arising in the least-squares collocation method for DAEs. In particular,
the solution algorithm has been split into a representation mapping
that connects the coefficients of the basis representation to the
function to be represented, and a linearly equality constrained linear
least-squares problem. A careful investigation of the representation
map allowed for a characterization of errors in the function spaces
by those made in the solution of the discrete problem.

The perturbation estimates for the constrained least-squares problem
have been derived with the application in mind: the approximation
of a DAE. The constraints play an exceptional role. If they are satisfied,
the resulting numerical solution belongs to the solution space $H_{D}^{1}(a,b)$.
If this cannot be guaranteed, the convergence theory for the least-squares
method does not apply. Some of the characterizing quantities could
be estimated analytically for reasonable choices of bases while others
have been estimated numerically in certain examples. We believe that
these considerations contribute to a robust and efficient implementation
of the proposed method, which seems to provide surprisingly accurate
numerical solutions to higher-index DAEs.

\paragraph*{Acknowledgment}

The author wants to thank Roswitha März for many discussions that
led to a great enhancement of the presentation. In particular, her
contributions simplified the proof of Theorem~\ref{th:Rnorm} considerably.

\appendix

\section{Projections\label{sec:Projections}}

In the case of a perturbed kernel $\ker\mathcal{C}$ of $\mathcal{C}$
the (perturbed) discrete solution $x=\mathcal{R}c$ does not necessarily
belong to $X_{\pi}\subset H_{D}^{1}(a,b)$. The derivations in this
section will provide a possibility to project the computed solution
$c$ onto $\ker\mathcal{C}$ such that it becomes the representation
of an element of $X_{\pi}$.

\subsection{Representation of scalar products}

In accordance with the definitions (\ref{eq:L-1}) and (\ref{eq:H}),
we equip $\tilde{X}_{\pi}$ with two scalar products,
\begin{align*}
(x,y)_{L^{2}} & =\sum_{j=1}^{n}\left\{ \sum_{\kappa=1}^{k}\int_{t_{j-1}}^{t_{j}}x_{j\kappa}(t)y_{j\kappa}(t)\dt t+\sum_{\kappa=k+1}^{m}\int_{t_{j-1}}^{t_{j}}x_{j\kappa}(t)y_{j\kappa}(t)\dt t\right\} \\
(x,y)_{H_{D,\pi}^{1}} & =\sum_{j=1}^{n}\left\{ \sum_{\kappa=1}^{k}\int_{t_{j-1}}^{t_{j}}\left(x_{j\kappa}(t)y_{j\kappa}(t)+x'_{j\kappa}(t)y'_{j\kappa}(t)\right)\dt t+\sum_{\kappa=k+1}^{m}\int_{t_{j-1}}^{t_{j}}x_{j\kappa}(t)y_{j\kappa}(t)\dt t\right\} 
\end{align*}
for $x,y\in\tilde{X}_{\pi}$. Let $c=\mathcal{R}^{-1}x$ and $d=\mathcal{R}^{-1}y$.
Define, for $\kappa=1,\ldots,k$,
\begin{align*}
X_{j\kappa} & =\left[\begin{array}{c}
x_{j\kappa}(\bar{\tau}_{j1})\\
\vdots\\
x_{j\kappa}(\bar{\tau}_{j,N+1})
\end{array}\right],\quad X'_{j\kappa}=\left[\begin{array}{c}
x'_{j\kappa}(\bar{\tau}_{j1})\\
\vdots\\
x'_{j\kappa}(\bar{\tau}_{j,N+1})
\end{array}\right],\\
Y_{j\kappa} & =\left[\begin{array}{c}
y_{j\kappa}(\bar{\tau}_{j1})\\
\vdots\\
y_{j\kappa}(\bar{\tau}_{j,N+1})
\end{array}\right],\quad Y'_{j\kappa}=\left[\begin{array}{c}
y'_{j\kappa}(\bar{\tau}_{j1})\\
\vdots\\
y'_{j\kappa}(\bar{\tau}_{j,N+1})
\end{array}\right],
\end{align*}
and, for $\kappa=k+1,\ldots,m$,
\[
X_{j\kappa}=\left[\begin{array}{c}
x_{j\kappa}(\tau_{j1})\\
\vdots\\
x_{j\kappa}(\tau_{jN})
\end{array}\right],\quad X_{j\kappa}=\left[\begin{array}{c}
x_{j\kappa}(\tau_{j1})\\
\vdots\\
x_{j\kappa}(\tau_{jN})
\end{array}\right].
\]
Then we have, for $\kappa=1,\ldots,k$,
\begin{align*}
\int_{t_{j-1}}^{t_{j}}x_{j\kappa}(t)y_{j\kappa}(t)\dt t & =h_{j}\sum_{i=1}^{N+1}\bar{\gamma}_{i}x_{j\kappa}(\bar{\tau}_{ji})y_{j\kappa}(\bar{\tau}_{ji})=h_{j}Y_{j\kappa}^{T}\bar{\Gamma}^{2}X_{j\kappa}\\
 & =h_{j}^{3}d_{j\kappa}^{T}\bar{V}^{T}\bar{\Gamma}^{2}\bar{V}c_{j\kappa}
\end{align*}
and
\begin{align*}
\int_{t_{j-1}}^{t_{j}}\left(x_{j\kappa}(t)y_{j\kappa}(t)+x'_{j\kappa}(t)y'_{j\kappa}(t)\right)\dt t & =h_{j}\sum_{i=1}^{N+1}\bar{\gamma}_{i}\left(x_{j\kappa}(\bar{\tau}_{ji})y_{j\kappa}(\bar{\tau}_{ji})+x_{j\kappa}(\bar{\tau}_{ji})y_{j\kappa}(\bar{\tau}_{ji})\right)\\
 & =h_{j}Y_{j\kappa}^{T}\bar{\Gamma}^{2}X_{j\kappa}+h_{j}(Y'_{j\kappa})^{T}\bar{\Gamma}^{2}X'_{j\kappa}\\
 & =h_{j}^{3}d_{j\kappa}^{T}\bar{V}^{T}\bar{\Gamma}^{2}\bar{V}c_{j\kappa}+h_{j}d_{j\kappa}^{T}\mathring{V}^{T}\bar{\Gamma}^{2}\mathring{V}c_{j\kappa}.
\end{align*}
Similarly, for $\kappa=k+1,\ldots,m$, we have
\[
\int_{t_{j-1}}^{t_{j}}x_{j\kappa}(t)y_{j\kappa}(t)\dt t=h_{j}\sum_{i=1}^{N}\gamma_{i}x_{j\kappa}(\tau_{ji})y_{j\kappa}(\tau_{ji})=h_{j}d_{j\kappa}^{T}V^{T}\Gamma^{2}Vc_{j\kappa}.
\]
Using the matrices $\mathcal{U}$ (\ref{eq:U}) and $\hat{\mathcal{U}}$
(\ref{eq:Uhat}), we arrive at the compact representations
\begin{equation}
(x,y)_{L^{2}}=d^{T}\mathcal{U}^{T}\mathcal{U}c=\langle\mathcal{U}d,\mathcal{U}c\rangle,\quad(x,y)_{H_{D,\pi}^{1}}=d^{T}\hat{\mathcal{U}}^{T}\hat{\mathcal{U}}c=\langle\hat{\mathcal{U}}d,\hat{\mathcal{U}}c\rangle.\label{eq:scapro-1}
\end{equation}

\subsection{The orthogonal projection in $\tilde{X}_{\pi}$}

We are interested in computing the best approximation of a function
$x\in\tilde{X}_{\pi}$ in $X_{\pi}$ both with respect to $\lVert\cdot\rVert_{L^{2}}$
and $\lVert\cdot\rVert_{H_{D,\pi}^{1}}$. 
\begin{prop}
Let $Q_{L^{2}}x$ and $Q_{H_{D,\pi}^{1}}x$ denote the orthogonal
projection of $x$ with respect to $\lVert\cdot\rVert_{L^{2}}$ and
$\lVert\cdot\rVert_{H_{D,\pi}^{1}}$, receptively. Let $c=\mathcal{R}^{-1}x$,
$c_{L^{2}}=\mathcal{R}^{-1}Q_{L^{2}}x$, and $c_{H_{D,\pi}^{1}}=\mathcal{R}^{-1}Q_{H_{D,\pi}^{1}}x$.
Then it holds
\begin{align*}
c_{L^{2}} & =c-(\mathcal{U}^{T}\mathcal{U})^{-1}\mathcal{C}^{T}\left(\mathcal{C}(\mathcal{U}^{T}\mathcal{U})^{-1}\mathcal{C}^{T}\right)^{-1}\mathcal{C}c,\\
c_{H_{D,\pi}^{1}} & =c-(\hat{\mathcal{U}}^{T}\hat{\mathcal{U}})^{-1}\mathcal{C}^{T}\left(\mathcal{C}(\hat{\mathcal{U}}^{T}\hat{\mathcal{U}})^{-1}\mathcal{C}^{T}\right)^{-1}\mathcal{C}c.
\end{align*}
\end{prop}
\begin{proof}
Consider $Q_{L^{2}}x$ first. Let $\mathcal{C}^{\dagger}$ denote
the generalized inverse of $\mathcal{C}$ with respect to the decomposition
$\R^{n(mN+k)}=\ker\mathcal{C}\oplus\ker\mathcal{C^{\perp_{\mathcal{U}}}}$
where $\perp_{\mathcal{U}}$ denotes the orthogonal complement with
respect to the scalar product $(\mathcal{U}\cdot,\mathcal{U}\cdot)$
in (\ref{eq:scapro-1}). Then, $c_{L^{2}}=(I_{\R^{n(mN+k)}}-\mathcal{C}^{\dagger}\mathcal{C})c$.
According to \cite[Section 4.5]{NaVo76}, in our setting, $\mathcal{X}=\mathcal{C}^{\dagger}$
is equivalent to the conditions
\begin{align*}
\mathcal{C}\mathcal{X}\mathcal{C} & =\mathcal{C},\\
\mathcal{X}\mathcal{C}\mathcal{X} & =\mathcal{X},\\
(\mathcal{XC})^{T} & =(\mathcal{U}^{T}\mathcal{U})\mathcal{XC}(\mathcal{U}^{T}\mathcal{U})^{-1},\\
(\mathcal{CX})^{T} & =\mathcal{CX}.
\end{align*}
Set $\mathcal{X}=(\mathcal{U}^{T}\mathcal{U})^{-1}\mathcal{C}^{T}\left(\mathcal{C}(\mathcal{U}^{T}\mathcal{U})^{-1}\mathcal{C}^{T}\right)^{-1}$.
The conditions are easily verified:
\begin{align*}
\mathcal{C}\mathcal{X}\mathcal{C} & =\mathcal{C}(\mathcal{U}^{T}\mathcal{U})^{-1}\mathcal{C}^{T}\left(\mathcal{C}(\mathcal{U}^{T}\mathcal{U})^{-1}\mathcal{C}^{T}\right)^{-1}\mathcal{C}=\mathcal{C},\\
\mathcal{X}\mathcal{C}\mathcal{X} & =(\mathcal{U}^{T}\mathcal{U})^{-1}\mathcal{C}^{T}\left(\mathcal{C}(\mathcal{U}^{T}\mathcal{U})^{-1}\mathcal{C}^{T}\right)^{-1}\mathcal{C}(\mathcal{U}^{T}\mathcal{U})^{-1}\mathcal{C}^{T}\left(\mathcal{C}(\mathcal{U}^{T}\mathcal{U})^{-1}\mathcal{C}^{T}\right)^{-1}\\
 & =(\mathcal{U}^{T}\mathcal{U})^{-1}\mathcal{C}^{T}\left(\mathcal{C}(\mathcal{U}^{T}\mathcal{U})^{-1}\mathcal{C}^{T}\right)^{-1}=\mathcal{X},\\
(\mathcal{XC})^{T} & =\left((\mathcal{U}^{T}\mathcal{U})^{-1}\mathcal{C}^{T}\left(\mathcal{C}(\mathcal{U}^{T}\mathcal{U})^{-1}\mathcal{C}^{T}\right)^{-1}\mathcal{C}\right)^{T}\\
 & =\mathcal{C}^{T}\left(\mathcal{C}(\mathcal{U}^{T}\mathcal{U})^{-1}\mathcal{C}^{T}\right)^{-1}\mathcal{C}(\mathcal{U}^{T}\mathcal{U})^{-1}\\
 & =(\mathcal{U}^{T}\mathcal{U})(\mathcal{U}^{T}\mathcal{U})^{-1}\mathcal{C}^{T}\left(\mathcal{C}(\mathcal{U}^{T}\mathcal{U})^{-1}\mathcal{C}^{T}\right)^{-1}\mathcal{C}(\mathcal{U}^{T}\mathcal{U})^{-1}\\
 & =(\mathcal{U}^{T}\mathcal{U})\mathcal{XC}(\mathcal{U}^{T}\mathcal{U})^{-1},\\
(\mathcal{CX})^{T} & =\left(\mathcal{C}(\mathcal{U}^{T}\mathcal{U})^{-1}\mathcal{C}^{T}\left(\mathcal{C}(\mathcal{U}^{T}\mathcal{U})^{-1}\mathcal{C}^{T}\right)^{-1}\right)^{T}=I=\mathcal{CX}.
\end{align*}
Hence, the proposition for $c_{L^{2}}$ follows. The other statement
is verified similarly with $\mathcal{U}$ replaced by $\hat{\mathcal{U}}$.
\end{proof}
\begin{rem}
Assume that we apply the orthogonal projection $Q_{H_{D,\pi}^{1}}$
on $\tilde{x}=\mathcal{R}(c+\Delta c)$ where $c+\Delta c$ is a solution
of the perturbed problem (\ref{eq:pertmini}). Then it holds obviously
$\lVert Q_{H_{D,\pi}^{1}}\tilde{x}-x_{\pi}\rVert_{H_{D,\pi}^{1}}=\lVert Q_{H_{D,\pi}^{1}}(\tilde{x}-x_{\pi})\rVert_{H_{D,\pi}^{1}}\leq\lVert\tilde{x}-x_{\pi}\rVert_{H_{D,\pi}^{1}}$.
It is, however, not guaranteed that this projected approximation is
a better approximation to the exact solution $x_{\ast}$ than $\tilde{x}$.\qed
\end{rem}

\subsection{The orthogonal projection in $\R^{n(mN+k)}$}

The appearance of the term $(\mathcal{U}^{T}\mathcal{U})^{-1}$ in
the representation of the projection in $\tilde{X}_{\pi}$ makes the
evaluation of $\mathcal{C}^{\dagger}$ hard. The orthogonal projection
in $\R^{n(mN+k)}$ has a much simpler representation depending only
on $\mathcal{C}$, which in turn has a very simple structure.

With $\mathcal{C}$ from (\ref{eq:discconstraint}) it holds $\mathcal{C}^{+}=\mathcal{C}^{T}\left(\mathcal{C}\mathcal{C}^{T}\right)^{-1}$
such that the orthogonal projection $\mathcal{Q}_{\pi}$ onto $\ker\mathcal{C}$
is $\mathcal{Q}_{\pi}=I_{\R^{n(mN+k)}}-\mathcal{C}^{+}\mathcal{C}$.
Using (\ref{eq:CCT-1}) we obtain,
\begin{align}
\mathcal{Q}_{\pi} & =I_{\R^{n(mN+k)}}-\mathcal{C}^{+}\mathcal{C}\nonumber \\
 & =I_{\R^{n(mN+k)}}-\Pi_{2}^{T}\left[\begin{array}{c}
I_{k}\otimes\mathcal{C}_{\textrm{s}}^{T}\\
\mathcal{O}_{\textrm{s}}^{T}
\end{array}\right]\Pi_{1}^{T}\Pi_{1}\left(I_{k}\otimes(\mathcal{C_{\textrm{s}}}\mathcal{C}_{\textrm{s}}^{T})^{-1}\right)\Pi_{1}^{T}\Pi_{1}\left[I_{k}\otimes\mathcal{C}_{\textrm{s}}\vert\mathcal{O}_{\textrm{s}}\right]\Pi_{2}\nonumber \\
 & =I_{\R^{n(mN+k)}}-\Pi_{2}^{T}\left[\begin{array}{c}
I_{k}\otimes\mathcal{C}_{\textrm{s}}^{T}\\
\mathcal{O}_{\textrm{s}}^{T}
\end{array}\right]\left(I_{k}\otimes(\mathcal{C_{\textrm{s}}}\mathcal{C}_{\textrm{s}}^{T})^{-1}\right)\left[I_{k}\otimes\mathcal{C}_{\textrm{s}}\vert\mathcal{O}_{\textrm{s}}\right]\Pi_{2}\nonumber \\
 & =I_{\R^{n(mN+k)}}-\Pi_{2}^{T}\left[\begin{array}{c}
(I_{k}\otimes\mathcal{C}_{\textrm{s}}^{T})\left(I_{k}\otimes(\mathcal{C_{\textrm{s}}}\mathcal{C}_{\textrm{s}}^{T})^{-1}\right)\\
\mathcal{O}_{\textrm{s}}^{T}
\end{array}\right]\left[I_{k}\otimes\mathcal{C}_{\textrm{s}}\vert\mathcal{O}_{\textrm{s}}\right]\Pi_{2}\nonumber \\
 & =I_{\R^{n(mN+k)}}-\Pi_{2}^{T}\left[\begin{array}{cc}
I_{k}\otimes\mathcal{C}_{\textrm{s}}^{T}(\mathcal{C_{\textrm{s}}}\mathcal{C}_{\textrm{s}}^{T})^{-1}\mathcal{C}_{\textrm{s}} & \mathcal{O}\\
\mathcal{O}^{T} & \mathcal{O}_{\textrm{s}}^{T}\mathcal{O}_{\textrm{s}}
\end{array}\right]\Pi_{2}\label{eq:Cproj-1}
\end{align}
where $\mathcal{O}$ denotes a zero matrix of the corresponding size.

In order to derive the algorithm for computing the projection of a
vector $c\in\R^{n(mN+k)}$ it is useful to consider the permutations
$\Pi_{1}$ and $\Pi_{2}$. The row permutation $\Pi_{1}$ separates
the equations belonging to the individual components $x_{1},\ldots,x_{k}$
(in that order). The column permutation $\Pi_{2}$ reorders the coefficients
such that the ones describing the differential components are taken
first, and then the ones belonging to the algebraic components. The
representation (\ref{eq:Cproj-1}) shows in particular that the coefficients
describing the algebraic components will not be changed by the projection.
Hence, the projection algorithm can be describes as follows:
\begin{enumerate}
\item For $\kappa=1,\ldots,k$ do:
\begin{enumerate}
\item Determine the coefficients $c^{\kappa}=(c_{1\kappa0},c_{1,\kappa1},\ldots,c_{1\kappa N},c_{2\kappa0},\ldots,c_{n\kappa N})^{T}$.
\item Evaluate $\mathcal{C}_{\textrm{s}}c^{\kappa}$.
\item Solve $(\mathcal{C_{\textrm{s}}}\mathcal{C}_{\textrm{s}}^{T})d^{\kappa}=\mathcal{C}_{\textrm{s}}c^{\kappa}$
with $\mathcal{C}_{s}\mathcal{C}_{s}^{T}$ from (\ref{eq:CsCsT-1}).
\item Set $c^{\kappa}=c^{\kappa}-\mathcal{C}_{\textrm{s}}^{T}d^{\kappa}$.
\item Replace $c^{\kappa}$ in $c$.
\end{enumerate}
\item The coefficients $c^{\kappa}$ for $\kappa=k+1,\ldots,m$ are unchanged.
\end{enumerate}
\begin{rem}
Let $x_{\pi}=\mathcal{R}c$ be the discrete solution of (\ref{eq:PhiM})
and $\tilde{x}=\mathcal{R}(c+\Delta c)$ with $c+\Delta c$ being
the solution of the perturbed problem (\ref{eq:pertmini}). Apply
the projection $\mathcal{Q}_{\pi}$ onto the coefficient vector $c+\Delta c$.
Then it holds, for $\hat{x}=\mathcal{R}\mathcal{Q}_{\pi}(c+\Delta c)\in X_{\pi}$,
\begin{align*}
\lVert\hat{x}-x_{\pi}\rVert_{H_{D,\pi}^{1}} & =\lVert\mathcal{R}\mathcal{Q}_{\pi}(c+\Delta c)-\mathcal{R}c\rVert_{H_{D,\pi}^{1}}=\lVert\mathcal{R}\mathcal{Q}_{\pi}(c+\Delta c)-\mathcal{R}\mathcal{Q}_{\pi}c\rVert_{H_{D,\pi}^{1}}\\
 & \leq\lVert\mathcal{R}\rVert_{\R^{n(mN+k)}\rightarrow H_{D,\pi}^{1}}\lvert\Delta c\rvert=\sigma_{\max}(\hat{\mathcal{U}})\lvert\Delta c\rvert.
\end{align*}
Similarly,
\[
\lVert\hat{x}-x_{\pi}\rVert_{L^{2}((a,b),\R^{m})}\leq\sigma_{\max}(\mathcal{U}).
\]

\qed
\end{rem}
\vspace{0pt}
\begin{rem}
The linear system to be solved in step 1(c) is a simple, low dimensional
triangular system. Moreover, it is independent of $\kappa$. The special
form of $f$ in $\mathcal{C_{\textrm{s}}}\mathcal{C}_{\textrm{s}}^{T}$
shows that it is always diagonally dominant. Proposition~\ref{propCcond}
in Section~\ref{subsec:CondRepr} shows that it is well conditioned
for standard bases.

In the case of a constant stepsize $h_{j}\equiv h$, the triangular
system simplifies considerably because $h$ cancels out in the evaluation
of $\mathcal{C}_{\textrm{s}}^{T}(\mathcal{C_{\textrm{s}}}\mathcal{C}_{\textrm{s}}^{T})^{-1}\mathcal{C}_{\textrm{s}}$.\qed
\end{rem}
\begin{example*}
In this example, we apply the projection $\mathcal{Q}_{\pi}$ to the
function $x_{p}\in\tilde{X}_{\pi}$ given by
\[
x_{p}(t)=\begin{cases}
0, & t\in[t_{j-1},t_{j}),\text{ \ensuremath{j} even},\\
1, & t\in[t_{j-1},t_{j}),\text{ \ensuremath{j} odd}.
\end{cases}
\]
Therefore, the jump of $x_{p}$ at $t_{j}$, $j=1,\ldots,n-1$ equals
1. The maximal jump after projection by using different basis representation
is a small multiple of the rounding unit for all tested cases $N\in\{3,5,10,20\}$ and $n\in\{10,20,40,80,160,320\}$.\qed
\end{example*}

\bibliographystyle{plain}
\bibliography{LSCMImpl}

\end{document}